\documentclass[12pt,oneside,english]{amsart}
\usepackage[T1]{fontenc}
\usepackage[latin9]{inputenc}
\usepackage{verbatim}
\usepackage{prettyref}
\usepackage{amsthm}
\usepackage{amstext}
\usepackage{amssymb}

\makeatletter
\numberwithin{equation}{section}
\numberwithin{figure}{section}
  \theoremstyle{remark}
  \newtheorem*{acknowledgement*}{Acknowledgement}
  \theoremstyle{plain}
  \newtheorem*{fact*}{Fact}
\theoremstyle{plain}
\newtheorem{thm}{Theorem}
 \theoremstyle{definition}
  \newtheorem{example}[thm]{Example}
  \theoremstyle{plain}
  \newtheorem*{assumption*}{Assumption}
  \theoremstyle{remark}
  \newtheorem{rem}[thm]{Remark}
  \theoremstyle{definition}
  \newtheorem{defn}[thm]{Definition}
  \theoremstyle{plain}
  \newtheorem{prop}[thm]{Proposition}
  \theoremstyle{remark}
  \newtheorem{notation}[thm]{Notation}
  \theoremstyle{plain}
  \newtheorem{lem}[thm]{Lemma}
  \theoremstyle{plain}
  \newtheorem{cor}[thm]{Corollary}

\usepackage{diagrams}

\newcommand{\R}{\mathbb{R}}

\newcommand{\C}{\mathbb{C}}
\newcommand{\N}{\mathbb{N}}
\newcommand{\T}{\mathbb{T}}

\newcommand{\xH}{{\mathcal{H}_\hbar}}

\newcommand{\Frame}{\mathbf{F}}
\newcommand{\Energy}{\mathcal{E}}
\newcommand{\SCint}{{\textrm {s.c.}}\, \int}

\newcommand{\xOp}{\mathrm{Op}}
\newcommand{\Four}{{\bf Four}}
\newcommand{\SFour}{{\bf SFour}}

\newcommand{\WF}{\operatorname{WF}}

\newcommand{\cat}[3]{{#1}_{#2}^{#3}}

\newcommand{\micro}{\bf{mic}}
\newcommand{\ext}{\bf{ext}}

\newcommand{\sympl}{\bf{Sympl}}

\newcommand{\Mic}{\cat{\bf Cot}{\micro}{\ext}}

\newcommand{\ExtSympl}{\cat{\sympl}{}{\ext}}

\newcommand{\im}{{\operatorname{Im}}}

\newcommand{\Cot}{{T}^*}

\newcommand{\id}{\operatorname{id}}

\newcommand{\graph}{\operatorname{gr}}
\newcommand{\stat}[1]{{\bf Stat}_{(#1)}}

\newcommand{\neutral}{\textbf{E}}

\newcommand{\e}{\textbf{e}}

\newcommand{\pr}{\operatorname{pr}}

\makeatother

\makeatother

\usepackage{babel}

\begin{document}

\title{Symplectic Microgeometry IV: 
Quantization
}

\author{Alberto S. Cattaneo}
\address{Institut f\"ur Mathematik, Universit\"at Z\"urich\\
Winterthurerstrasse 190, CH-8057 Z\"urich, Switzerland}  
\email{cattaneo@math.uzh.ch}

\author{Benoit Dherin}
\address{Google, Clanwilliam Pl, Grand Canal Dock, Dublin, Ireland}
\email{dherin@google.com}

\author{Alan Weinstein}
\address{Department of Mathematics, University of California, Berkeley CA 94720 USA and Department of Mathematics, Stanford University, Stanford, CA 94305 USA}
\email{alanw@math.berkeley.edu}

\begin{abstract}
We construct a special class of semiclassical Fourier integral operators
whose wave fronts are the symplectic micromorphisms of \cite{SM_Microfolds}.
These operators have very good properties: they form a category on
which the wave front map becomes a functor into the cotangent microbundle
category, and they admit a total symbol calculus in terms of symplectic
micromorphisms enhanced with half-density germs. This new operator
category encompasses the semi-classical pseudo-differential calculus
and offers a functorial framework for the semi-classical analysis
of the Schr\"odinger equation. We also comment on applications to classical
and quantum mechanics as well as to a functorial and geometrical approach
to the quantization of Poisson manifolds. 
\end{abstract}
\maketitle
\tableofcontents{}
\begin{acknowledgement*}
The research of A.S.C. was partially supported by the NCCR SwissMAP, funded by the Swiss National Science Foundation, and by
SNF Grant No.\  200020\_192080. 
B.D. acknowledges partial support from the NWO Grant 613.000.602
and the SNF Grant PA002-113136. A.W. acknowledges partial support
from NSF grant DMS-0707137 and the UC Berkeley Committee on Research. We would like to thank David Barrett, Ivan Contreras and  
 Pedro de M. Rios for useful comments on this paper and discussions on generating
families, and Ted Voronov for extensive correspondence concerning his work on thick morphisms.  The latter began  around 2014, when we had prepared a preliminary version of the present paper, and Voronov had  posted on the ArXiv preprints of work (see citations in the introduction below) published several years later.
\end{acknowledgement*}

\section{Introduction}

From an analytical point of view, symplectic geometry is the geometry
underlying the calculus of Fourier integral operators (FIO's) \cite{BW1997,Duistermaat1996,DH1972,GS1977,Treves1981}.
The present article is concerned with developing a similar calculus
in the context of microsymplectic geometry \cite{SM_Microfolds}.
Since the latter enjoys much better functorial properties than its
macro analog, it is not too surprising that these good properties
persist in the associated operator calculus.

The geometry of canonical relation composition in the symplectic
``category'' is central to symplectic geometry itself in many
ways \cite{Weinstein1977}. This ``geometric calculus'' also
plays an important role in the calculus of FIO's through a  canonical relation
associated to each FIO: its wave front. The geometry of the wave front
usually contains a lot of information about the class of FIO's associated
to it. The reader will find detailed treatments of this relationship
in the standard literature on FIO's \cite{Duistermaat1996,DH1972,GS1977,Hormander1971,Treves1981}).
Here, we are mostly concerned with functorial and categorical aspects
of this relationship between the FIO's and their wave fronts. Namely,
just as for canonical relations, the composition of two
FIO's usually fails to be a FIO. Now, a well-behaved composition of
the wave front canonical relations, if some restriction is imposed on them, can guarantee a well-behaved composition
of the operators in the corresponding FIO classes. In this sense,
the symplectic geometry of the wave fronts controls the categorical
and functorial aspects of the FIO calculus.

In \cite{SM_Microfolds}, we constructed a category of particularly
well-behaved germs of canonical relations: the symplectic micromorphisms.
This article studies the corresponding class of (semi-classical) FIO's,
which inherits, for the most part, the good properties of the symplectic
micromorphism composition. It turns out that this class of operators
encompasses and extends the whole calculus of pseudo-differential operators
in the semi-classical limit. Moreover, the hamiltonian flows of classical
dynamics can be described, for asymptotically small times, in terms
of some symplectic micromorphisms satisfying purely algebraic relations
mimicking  time translations at a purely categorical level.
Upon quantization, we recover the Schr\"odinger flow of quantum mechanics
in the semi-classical limit in terms of our semi-classical FIO's. 
This approach to quantization is very related to the considerations surrounding the notion of ``thick morphism" developed by Voronov
in \cite{voronovthick,voronovpullback,voronovmicroformal} and by Khudaverdian and Voronov in \cite{ThickbbMorphisms}.  Their work, done in the context of supermanifolds, was originally aimed at an understanding of $L_\infty$ morphisms between homotopy structures, but many of their results are parallel to ours when restricted to the case of ordinary manifolds.
In addition, many of our results have  also been used in \cite{ActQuant, GSystems, MomMap} by Mencattini and one of the authors to quantize momentum maps as well
as the underlying group actions. In \cite{Leibniz}, it has been used by Wagemann and one of the authors to quantize Leibniz algebras.

One defining trait of Fourier integral operators is that they admit
explicit representations in terms of oscillatory integrals, once given
a generating family for their wave fronts. Symplectic micromorphisms
possess canonical generating families (up to a choice of an exponential
map germ on the smooth manifolds) with very good properties (stated
in paragraph \ref{sub:Integral-representation}) which allow one to
define a total symbol calculus for their corresponding semi-classical
FIO's. These canonical generating families are obtained as deformations
of generating families for cotangent lifts. Locally, such a semiclassical FIO
\begin{eqnarray*}
Q_{\hbar}(a,f,\phi):L_{\hbar}^{2}(\R^{m}) & \longrightarrow & L_{\hbar}^{2}(\R^{n})
\end{eqnarray*}
may be written in integral form similar to that for pseudo-differential
operators
\begin{equation}\label{eq:Loc-Rep}
\big(Q_{\hbar}u\big)(y)  :=  (2\pi\hbar)^{-\frac{m+n}{2}}\int a(p,y)e^{\frac{i}{\hbar}\big(\langle p,\phi(y)-x\rangle+f(p,y)\big)}u(x)\,dp\; dx,
\end{equation}
where $\phi:\R^{n}\rightarrow\R^{m}$ is a smooth map, $a$ is an
$\hbar$-dependent smooth function (on the pullback bundle $\phi^*(\Cot\R^{m})$) called the symbol of the operator,
and the phase is the generating function of a symplectic micromorphism
 from $\Cot\R^{m}$ to $\Cot\R^{n}$ with core map $\phi$. (See section \ref{sub:The-semiclassical-intrinsic} for a
precise definition of the semiclassical intrinsic Hilbert space $L_{\hbar}^{2}(\R^{n})$.) We recover pseudo-differential
operators when $m=n$, $\phi=\id$ and $f=0$, in which case the
corresponding symplectic micromorphism is the identity.

In this paper, we will mostly focus on the semi-classical limit of such operators, using the integral symbol "$\SCint $" instead of "$\int$" to remind us of this fact. The semi-classical limit is concerned with equivalence classes of such operators that have the same asymptotic behavior when the parameter
$\hbar$ in the phase of the oscillatory integral goes to zero. Depending
on the problem at hand, one may be interested in asymptotics of order
$\hbar^{N}$ for some fixed $N$; here, we will be interested
in asymptotics modulo $\hbar^{\infty}$, which roughly means that
we consider two FIO's equivalent if they have the same complete asymptotic
expansions in $\hbar$ as $\hbar\rightarrow0$ . 

An application of the stationary phase principle, Theorem \eqref{eq:Loc-Rep}
shows that the semi-classical limit of an FIO is controlled by the germ of
its wave front $\WF(Q_{\hbar})$ and of its symbol around $\graph\phi$ (the graph of
the symplectic micromorphism core map).  The symplectic microgeometry
terminology introduced in \cite{SM_Microfolds} is a very convenient
language to deal with the semi-classical limit of these FIO's.

Let us recall some basic definitions of microgeometry
as introduced in \cite{SM_Microfolds}. A \textbf{microfold} 
$[M,A]$ (called a ``local manifold pair" in \cite{Weinstein1971}) is an equivalence class of manifold pairs $(M,A)$, where
$A$ is a closed submanifold of $M$ such that two pairs $(M_{1},A)$
and $(M_{2},A)$ are equivalent if there exists a third one $(M_{3},A)$
for which $M_{3}$ is simultaneously an open submanifold of both $M_{1}$
and $M_{2}$. A \textbf{symplectic microfold} is a microfold $[M,A]$
such that $M$ is a symplectic manifold and $A$ is a lagrangian submanifold;
a \textbf{lagrangian submicrofold} $[L,S]$ of $[M,A]$ is a submicrofold
(i.e. a microfold such that $L\subset M$ and $S\subset A$) such
that $L$ is a lagrangian submanifold. 

This paper initiates the study of Fourier integral operators for which the semi-classical limit is controlled by a lagrangian submicrofold
of its wave front that satisfies the transversality condition in
\cite{SM_Microfolds}, or, in other words, by a symplectic micromorphism. 

\subsection*{Outline of the paper}

In Section \ref{sec:Categories-of-Fourier}, we give a brief introduction
to the calculus of semi-classical Fourier integral operators. In particular,
we review the central notion of generating families for lagrangian
submanifolds. Example \ref{sub:ConormBundle} recalls a standard construction
for generating families of conormal bundles which depends only on
a choice of a tubular neighborhood. This is the central construction
on which we are going to build on in this paper. We also stress the
functorial aspects of the FIO calculus, and we introduce the notion
of a {}``Fourier category'' associated to any always-well-composing collection of 
wave fronts. In paragraph \ref{sub:Integral-representation}, we isolate
special conditions for the wave fronts in terms of their generating
families guaranteeing a very convenient integral representation for
their associated FIO's. 

In Section \ref{sec:Quantization-of-cotangent}, we focus on semi-classical
FIO's whose wave fronts are cotangent lifts of smooth maps. These
canonical relations form a category, and so do their associated
FIO's. We show how to use the standard generating family construction
for conormal bundles in the context of cotangent lifts using exponential
map germs to construct the tubular neighborhood required by the construction.
We show that the domain of such a  generating family admits a canonical
vertical half-density. This fact is very important to us,
since it reduces the usual ambiguity of the FIO integral representation
to only one thing: the choice of exponential germs on the underlying
manifolds. 

In Section \ref{sec:Quantization-of-symplectic}, we extend the class
of allowed wave fronts from cotangent lifts to all symplectic micromorphisms.
To do so, we first introduce a notion of deformation for conormal
bundles and their generating families. Then, we show that symplectic
micromorphisms are in one-to-one correspondence with deformations
of cotangent lifts generating families once an exponential map germ
has been fixed. This allows us to extend the integral representation
of Section \ref{sec:Quantization-of-cotangent} to FIO's whose wave
fronts are general symplectic micromorphisms. We study the local theory
of these operators, with several examples; in particular, we show
that semi-classical pseudo-differential operators fall into our class.
Moreover, we give explicit formulas for the operator composition in
this local setting; this involves a composition formula for the wave
front generating families as well as for the total symbols of these
operators. 

The presence of a canonical exponential in the local setting allows
us to identify our category of FIO's with the category of enhanced
symplectic micromorphisms between cotangent bundles of $\R^{n}$ for
$n=1,2,\ldots$,
where an enhanced symplectic micromorphism is a symplectic micromorphism
carrying near its core a half-density germ corresponding to the total
symbol of the operator. Namely, in this context, we have two inverse
functors: the quantization functor $Q_{\hbar}$ that associates to
an enhanced symplectic micromorphism a semi-classical FIO through
the integral representation as in \eqref{eq:Loc-Rep} and the total
symbol functor $\sigma$ that associates to the operator \eqref{eq:Loc-Rep}
its wave front (in the form of its generating function $f$) enhanced
with the total symbol $a$ (identified in the local case with a smooth
function germ on $\phi^{*}(\Cot\R^{m})$ around the zero section). 

In Section \ref{sec:Applications-and-further}, we comment on applications
and further directions. In particular, we explain how to extend the
quantization and total symbol functors, which are present in the local
setting, to a semi-classical FIO calculus over any smooth manifold
category enriched with some additional geometric structures sufficient
to allow the construction of exponential map germs on the manifolds
in a canonical way. We continue by relating the calculus developed
here to the quantization of Poisson manifolds via oscillatory integrals
and symplectic groupoids \cite{Karasev1986,Weinstein1990,Zakrzewski1990}.
(In particular the generating function of the formal symplectic groupoid
integrating a Poisson structure obtained in \cite{CDF2005} can be understood as the
jet of the generating function of the quantizing FIO in this setting.)
Finally, we show how the small-time asymptotics of classical hamiltonian
mechanics can be expressed in a purely categorical way in the framework
of symplectic microgeometry. In particular, we show how classical
flows can be modeled as the action of a special monoid in the microsymplectic
category, the energy monoid, on the hamiltonian system's phase-space;
enhancement and quantization of the action symplectic micromorphism
recovers the usual unitary Schr\"odinger flows of quantum mechanics.
We also explain how symmetries can be modeled very naturally in this
framework, both at the classical and quantum level.

\section{Categories of Fourier integral operators\label{sec:Categories-of-Fourier}}

In this section, we give a brief presentation of the theory of Fourier
integral operators. We focus mostly on the categorical and geometrical
aspects of this calculus since they are the main concern for us in
this paper, referring the reader to standard texts \cite{Duistermaat1996,DH1972,GS1977,Hormander1971,Treves1981}
on the subject for the more analytical aspects.

\subsection{Categories of canonical relations}

\subsubsection{The symplectic ``category''}

Let $M_{i}$, $i=1,2,3$, be three symplectic manifolds, and let $L_{i}$
($i=1,2$) be a canonical relation from $M_{i}$ to $M_{i+1}$, i.e.,
a closed lagrangian submanifold of the symplectic manifold product $\overline{M}_{i}\times M_{i+1}$,
where $\overline{M}$ denotes the symplectic manifold with opposite
symplectic form $-\omega_{M}$. One can compose $L_{1}$ with $L_{2}$
as binary relations yielding the subset $L_{2}\circ L_{1}$ of $\overline{M_{1}}\times M_{3}$.
This composition fails in general to be a lagrangian submanifold,
or even a submanifold; however, there are many examples for which
it does. For instance, we have the well know proposition:

\begin{prop}
 A sufficient condition for the set-theoretic composition of the canonical relations $L_{1}$ and $L_{2}$ to be a canonical relation
 is that the intersection in $\overline{M_1} \times M_2 \times \overline{M_2}
 \times M_3$ of $L_{1}\times L_{2}$
with $M_{1}\times\Delta_{M_{2}}\times M_{3}$ (where $\Delta_M$ denotes the diagonal in $M\times M$) is transversal and properly
embedded in $M_{1}\times M_{3}$ via the canonical factor projection.
In this case, we say that the canonical relations have \textbf{strongly
tranverse composition}. 
\end{prop}
We denote by $\ExtSympl$ the (extended) symplectic ``category''
whose objects are cotangent bundles and whose morphisms are taken
to be canonical relations between them. The quotation marks are there to stress
that this is not a category in the usual sense since composition is not
always defined. However, this ``category'' contains a honest subcategory
formed by the cotangent lifts as described below.

\subsubsection{Schwartz transform and cotangent lifts}

The canonical relations from $\Cot M$ to $\Cot N$ can be put in
one-to-one correspondence with the lagrangian submanifolds of $\Cot(M\times N)$
via the Schwartz transform (see \cite{BW1997}),\begin{eqnarray*}
\mathcal{S}:\overline{\Cot M}\times\Cot N & \longrightarrow & \Cot(M\times N),\end{eqnarray*}
which is the symplectomorphism that sends $\big((p_{1},x_{1}),\,(p_{2},x_{2})\big)$
to $(-p_{1},p_{2},x_{1},x_{2})$. Now, to any smooth map $M\overset{\phi}{\leftarrow}N$,
we can associate a special canonical relation, its cotangent lift,
which we will consider going in the opposite direction to $\phi$,\begin{eqnarray*}
\Cot\phi:\Cot M & \longrightarrow & \Cot N,\end{eqnarray*}
by pulling back the conormal bundle%
\footnote{Recall that the conormal bundle $N^{*}S$ of a submanifold $S\subset X$
is the lagrangian submanifold of $\Cot X$ consisting of the covectors to
$X$ based along S and vanishing on $TS$.%
} of $\graph\phi$, seen as a submanifold of $\Cot(M\times N)$, via
the Schwartz transform:\[
\Cot\phi:=\Big\{\Big(\big(p_{1},\,\phi(x_{2})\big),\,\big((T_{x_{2}}\phi)^{*}p_{1},\, x_{2}\big)\Big):\,(p_{1},x_{2})\in\phi^{*}(\Cot M)\Big\}.\]
The collection $\mathcal{C}$ of all cotangent lifts is a subcategory
of $\ExtSympl$ which is a true category; namely, we always have that\[
\Cot\phi_{2}\circ\Cot\phi_{1}=\Cot(\phi_{1}\circ\phi_{2}).\]

\subsection{Generating families}

\subsubsection{Generating functions}

An \textbf{exact lagrangian embedding} $\lambda:\Sigma\hookrightarrow\Cot X$
is a lagrangian embedding for which $\lambda^{*}\theta=dS$ for some
$S\in C^{\infty}(\Sigma)$ , where $\theta$ is the Liouville $\mbox{1}$-form
on $\Cot X$. 

Composing $\lambda$ with the bundle projection $\pi:\Cot X\rightarrow X$,
we obtain a map $\pi_{\Sigma}$ from $\Sigma$ to $X$. \textbf{Antecaustic
points} are elements of $\Sigma$ at which $T\pi_{\Sigma}$ is not an isomorphism.\footnote{We have chosen this term because the images of these points under $\pi_\Sigma$ are known as caustic points.  (Although the antecaustic points play a key role in symplectic geometry, we have not found another concise term for them in the literature.)} 

When $\pi_{\Sigma}$ is a diffeomorphism, we say that the lagrangian
submanifold $\lambda(\Sigma)$ is \textbf{projectable}, in which case
the differential $dS_{\Sigma}:X\rightarrow\Cot X$ of $S_{\Sigma}:=S\circ\pi_{\Sigma}^{-1}$
parametrizes the lagrangian submanifold $\lambda(\Sigma)$. The function
$S$ is called a \textbf{generating function} of the lagrangian submanifold.
(It is well-defined up to a constant on each component of $\Sigma.$)
Conversely, to any function $S\in C^{\infty}(X)$, we can associate
the projectable lagrangian submanifold $\im~dS\subset\Cot X$ that
has $S$ as a generating function; explicitly,\[
\im~dS:=\big\{(dS(x),x):\, x\in X\big\}.\]

There are, however, many interesting non-projectable lagrangian submanifolds.
For instance, the conormal bundle $N^{*}C$ of a (non-open) submanifold $C\subset X$
is highly non-projectable in the sense that all points are antecaustic
points: if we consider the lagrangian embedding given by the inclusion\[
\iota_{C}:N^{*}C\longrightarrow\Cot X,\]
we see that the preimage of $\pi\circ\iota_{C}$ at any point $c\in C$
consists of the whole fiber $N_{c}^{*}C$. As a consequence cotangent 
lifts are also non-projectable, since they are conormal bundles to the graph of
the underlying map.

For these lagrangian submanifolds with antecaustic points, there is still
a notion of generating function. The price to pay, however, is the
introduction of additional variables for the generating functions
through a fibration $p:B\rightarrow X$ that {}``unfolds'' the lagrangian
submanifold at antecaustic points. This leads to the notion of generating
family.

\subsubsection{Generating families}

A submersion  $p:B\rightarrow X$ together with a smooth function $S\in C^{\infty}(B)$
defines two lagrangian submanifolds: the lagrangian submanifold $\im~ dS$
in $\Cot B$ whose generating function is $S$ (we regard it as a
canonical relation from the point to $\Cot B$) and the cotangent
lift $\Cot p$, which we can see as a canonical relation from $\Cot B$
to $\Cot X$. If these lagrangian submanifolds have a strongly transversal
composition\footnote{Recall that the composition is called {\bf transversal} if
the product $\im~dS \times T^*p$ of these canonical relations has transversal intersection with the diagonal
$\Delta_{T^*B}\times T^*X$ in $T^*B \times \overline{T^*B} \times T^*X$ (we ignore the point in the definition of $\im~dS$), and {\bf strongly transversal} if the image of this transversal intersection embeds properly in $T^*X$ under the natural projection. When the composition is transversal, $S$ is also called a {\bf Morse family} of functions over $X$.} their composition is a lagrangian submanifold of $\Cot X$.  
A \textbf{generating family} for a lagrangian submanifold $L$ in
$\Cot X$ is a triple $(B,p,S)$ as above such that\[
L=\Cot p\circ\im~dS.\]
Given a function $S\in C^{\infty}(B)$ and a fibration $p:B\rightarrow X$,
the canonical relations $\im~dS$ and $\Cot p$ 
have a strongly transversal composition if and only if the
first factor projection $\pi_{B}$ of $\Cot p$ on $\Cot B$ is transversal
to $\im~dS$ and the second factor projection $\pi_{X}$ of $\Cot p$
on $\Cot X$ becomes a proper embedding when restricted to the points
$y\in\Cot p$ such that $\pi_{B}(y)\in\im~dS$. Observe that the
intersection $\im~\pi_{B}\cap\im~dS$ consists of the points such
that\begin{equation}
(dS(b),b)=(\Cot_{b}p(\eta),b),\label{eq:crit sub cond}\end{equation}
for some $\eta\in\Cot_{\pi(b)}X$. Since $\im~\pi_{B}$ is the annihilator
of the vertical bundle of the fibration $p$ (i.e. all covectors vanishing
on the subbundle $\ker p_{*}$ of $TB$), a point of $(dS(b),b)$
is in $\im~\pi_{B}$ iff the vertical part of $dS$ vanishes%
\footnote{$i^{*}\circ dS(b)=0$, where $i^{*}$ is the dual of the inclusion
$i$ of the vertical subbundle $\ker p_{*}$ into $TB$. %
} at $b$. In the  case of strongly transversal composition, the set $\Sigma$
of all points in $B$ where this happens is a submanifold, called
the \textbf{fiber} \textbf{critical submanifold} of the generating
family. The smooth map\begin{eqnarray*}
\lambda:\Sigma & \longrightarrow & \Cot X\\
b & \mapsto & (\eta,p(b))\end{eqnarray*}
with $\eta$  defined by equation \eqref{eq:crit sub cond} is
a lagrangian embedding whose image is exactly $\Cot p\circ\im~dS$.

Since the standard construction of a generating family for a conormal
bundle $N^{*}C$ out of a tubular neighborhood of $C\subset X$, is
central for us, we spell it out in the following example (\cite{BW1997}). 
\begin{example}
\label{sub:ConormBundle}To begin, we fix a \textbf{tubular neighborhood}
of $C$, that is, the data $(V,\Psi)$ of a neighborhood $V$ of $C$
in $X$ equipped with a diffeomorphism $\Psi$ from a neighborhood
$U$ of the zero section of the normal bundle $NC$ into $V$ that
maps the zero section identically to $C$. We denote by $U_{c}$ the
restriction of $U$ to the fiber $N_{c}C$ and, correspondingly, by
$\Psi_{c}$ the restriction of $\Psi$ to $U_{c}$. This allows us
to map a neighborhood of the zero section of the bundle $N^{*}C\oplus NC$
over $C$ diffeomorphically into an open submanifold of $N^{*}C\times X$
as follows:\begin{eqnarray*}
(p,v,c) & \longmapsto & (p,\Psi_{c}(v)),\end{eqnarray*}
where $p\in N_{c}^{*}C$ and $v\in N_{c}C$. This gives us a generating
family $(B_{C}^{\Psi},p_{C},S_{C})$ for $N^{*}C$, where $B_{C}^{\Psi}$
is defined as the image of this mapping. We will denote by $(p,x,c)$
points in $B_{C}^{\Psi}$. The fibration $p_{C}:B_{C}^{\Psi}\rightarrow V$
is the projection $(p,x,c)\mapsto x$, and the generating function
is given by the canonical pairing\begin{eqnarray*}
S_{C}(p,x,c) & = & \big\langle p,\Psi_{c}^{-1}(x)\big\rangle.\end{eqnarray*}
The critical submanifold $\Sigma_{C}$ of $S_{C}$ then consists of
the points of the form $(p,c,c)$ where $c\in C$ and $p\in N_{c}C$.
This yields an embedding\[
\tau_{C}:N^{*}C\longrightarrow B_{C}^{\Psi},\]
whose image is exactly $\Sigma_{C}$, which has an obvious retraction
$r_{C}(p,x,c)=(p,c)$. Composing $\tau_{C}$ with the lagrangian embedding\[
\lambda_{C}:\Sigma_{C}\longrightarrow\Cot X,\]
generated by the generating family, we obtain the canonical inclusion
$\iota_{C}$ of $N^{*}C$ into $T^{*}X$. 
\end{example}

\subsection{Half-densities}

\subsubsection{$\alpha$-densities and the intrinsic Hilbert space}

An $\alpha$-density for $\alpha\in\C$ on a $n$-dimensional real
vector space $V$ is a  map $\rho$ from the space of frames
$\Frame(V)$ (a frame is an ordered basis $\e=(e_{1},\dots,e_{n})$)
to $\C$ such that\begin{eqnarray*}
\rho(\e\cdot A) & = & |\det A|^{\alpha}\rho(\e)\end{eqnarray*}
for any matrix $A\in GL(n)$. We denote by $|V|^{\alpha}$ the one-dimensional
complex vector space of $\alpha$-densities on $V$.  

Given a finite dimensional vector bundle $E$ over a smooth manifold
$X$, we denote by $|E|^{\alpha}$ the complex line bundle over $M$
whose fiber at $x\in X$ is $|E_{x}|^{\alpha}$. We will reserve the
notation $|\Omega|^{\alpha}(X)$ for the space of smooth sections
of $|TM|^{\alpha}$; its subspace of compactly supported sections
will be denoted by $|\Omega|_{c}^{\alpha}(X)$. For general density
bundles $|E|^{\alpha}\rightarrow X$, we will use the standard notation
$\Gamma(X,|E|^{\alpha})$ to denote the section space. 

An exact sequence of vector bundles\[
0\longrightarrow A\longrightarrow B\longrightarrow C\longrightarrow0\]
over a manifold $X$ induces a canonical isomorphism between the density
bundles $|B|^{\alpha}\simeq|A|^{\alpha}\otimes|C|^{\alpha}$ as well
as their corresponding section spaces. In particular, we can identify 
$|\Omega|^{\frac{1}{2}}(M)\otimes|\Omega|^{\frac{1}{2}}(M)$ with
the space $|\Omega|^{1}(M)$ of $1$-density bundle sections on $M$.
These sections can be integrated, which allows us to give to the vector
space $|\Omega|_{c}^{\frac{1}{2}}(M)$ of compactly supported half-densities
the structure of a pre-Hilbert space with the symmetric bilinear form\begin{eqnarray*}
\langle\mu,\nu\rangle & := & \int_{M}\overline{\mu}\nu,\end{eqnarray*}
where $\bar{\mu}$ is the complex conjugate of the half-density $\mu$.
The completion of this pre-Hilbert space is usually called the intrinsic
Hilbert space of $M$ (see \cite{BW1997} for more details), and we
will denote it by $\mathcal{H}(M)$.

\subsubsection{Integrating half-densities over fibrations \label{sub:Averaging-half-densities}}

Let $p:B\rightarrow X$ be a fibration. There is a notion of {}``pushforward''
from half-densities on $B$ to half-densities on $X$ that requires
some extra data in the form of a section of $|\ker p_{*}|^{\frac{1}{2}}\rightarrow X$,
where $\ker p_{*}$ the vertical bundle of the fibration. It goes
as follows: First observe that the exact sequence\[
0\longrightarrow\ker p_{*}\longrightarrow TB\longrightarrow p^{*}(TX)\longrightarrow0,\]
of vector bundles over $B$ induces the canonical isomorphism\begin{eqnarray}
|\Omega|^{\frac{1}{2}}(B) & \simeq & \Gamma(B,|\ker p_{*}|^{\frac{1}{2}})\otimes\Gamma(B,|p^{*}(TX)|^{\frac{1}{2}}).\label{eq: can iso B}\end{eqnarray}
Now, suppose that we are given a smooth family of half-densities $\rho(x)\in|\Omega|^{\frac{1}{2}}(p^{-1}(x))$
compactly supported on the fibers of $p$. The data of $\rho$ allows us to map half-densities
on $B$ to half-densities of $M$ by integrating them on the fibers
of $p$. Namely, $\rho$ can be regarded as a section of the half
density bundle $|\ker p_{*}|^{\frac{1}{2}}\rightarrow X$, and, in
the light of \eqref{eq: can iso B}, we can regard the tensor product
$\mu\otimes\rho$ for any half-density $\mu$ on $B$ as living in
\[
\Gamma(B,|\ker p_{*}|^{1})\otimes\Gamma(B,|p^{*}(TX)|^{\frac{1}{2}}).\]
The restriction of $\mu\otimes\rho$ to the fibers of $p$ gives thus
a family of densities on the fiber $p^{-1}(x)$ with values in the
fixed vector space $|T_{x}X|^{\frac{1}{2}}$. Therefore, we can integrate $\mu\otimes\rho$
on this fibers and obtain\begin{eqnarray*}
\theta(x) & := & \int_{p^{-1}(x)}\mu\otimes\rho\in|T_{x}X|^{\frac{1}{2}},\end{eqnarray*}
which is a half-density on $X$.

\subsection{Fourier integral operators}

We describe here special classes of FIO that are of concern for us.
For a more general presentation, we refer the reader to the standard
references \cite{DH1972,Duistermaat1996,Hormander1971}. We 
begin by outlining the main ingredients out of which FIO's are made and
by commenting on the general problem of FIO composition, which parallels
the ill-defined composition of canonical relations.

\subsubsection{The FIO {}``category'' }

Given a canonical relation $L$ from $\Cot X$ to $\Cot Y$, one can associate a class of operators $\Four_{\hbar}(L;X,Y)$,
the Fourier integral operators with \textbf{wave front} $L$, from
the intrinsic Hilbert space $\mathcal{H}(X)$ to the intrinsic Hilbert space
$\mathcal{H}(Y)$. More precisely, an operator $Q$ in this class is 
a family $Q=\{Q_{\hbar}:\,\hbar\in (0,1]\}$ of operators depending
smoothly on a parameter $\hbar$ in a way that will be made clear
later on. We will be mostly concerned with the asymptotics of these
operators in the \textbf{semiclassical limit, }that is, when $\hbar\rightarrow0$. 
For us, a \textbf{semiclassical FIO} will mean an equivalence class
of FIO's that have the same expansion in $\hbar$ at all orders. The space of all FIO's
between $X$ and $Y$ will be denoted by $\Four_{\hbar}(X,Y)$ and
the space of semiclassical FIO's by $\SFour(X,Y)$. We will come back
to the latter at the end of this section. 

An FIO can be defined explicitly in terms of an oscillatory integral
representation. For this, we need to fix a generating family for the
wave front as well as some {}``vertical'' half-density on the total
space of the generating family. We refer the reader to \cite{GS1977}
for the general description of these representations. 

In general, the composition of two FIO's fails to be a FIO and, therefore,
the collection $\Four_{\hbar}$ of all FIO's only form a {}``category''
in the same sense as $\ExtSympl$ does. Actually, this is more than
an mere analogy. Namely, we can define the map\[
WF:\Four_{\hbar}(X,Y)\longrightarrow\ExtSympl(\Cot X,\Cot Y)\]
that associates to an FIO $T$ its wave front $\WF(T)$. Now, it is
a well know-result (\cite{GS1977}) that when the wave fronts of two
FIO's $T_{1}$ and $T_{2}$ have strongly transversal composition, then the
composition of the operators themselves yields a FIO whose wave front
is given by the composition of the wave fronts: \[
\WF(T_{2}\circ T_{1})=\WF(T_{2})\circ\WF(T_{1}).\]
Moreover, the wave front generating families also compose: if $(B_{i},p_{i},S_{i})$
is a wave front generating family for $T_{i}\in\Four_{\hbar}(M_{i},M_{i+1})$
$(i=1,2)$, then the fibration\[
B_{1}\times_{M_{2}}B_{2}\longrightarrow M_{1}\times M_{3}\]
together with the generating function\begin{eqnarray*}
(S_{1}+S_{2})(b_{1},b_{2}) & = & S_{1}(b_{1})+S_{2}(b_{2})\end{eqnarray*}
is the generating family for $\WF(T_{2}\circ T_{1})$. Therefore,
whenever we can find a subcategory $\mathcal{L}$ of canonical relations
in $\ExtSympl$ with strongly transversal composition, we obtain an honest category
$\Four_{\hbar}^{\mathcal{L}}$ of operators formed by the FIO's having
canonical relations in $\mathcal{L}$ as wave front. In this case
the wave front map becomes a functor\[
\WF:\Four_{\hbar}^{\mathcal{L}}\longrightarrow\mathcal{L}\]
between the corresponding categories. 
\begin{example}
For instance, the category $\mathcal{C}$ of cotangent lifts has an
associated category of FIO's, for which the \textquotedbl{}kernel\textquotedbl{}
of the wave front functor $\WF$ assigns to each identity morphism
$\Cot\id_{X}$ in $\mathcal{C}$ the algebra of pseudodifferential
operators on $X$. 
\end{example}

\subsubsection{Integral representation\label{sub:Integral-representation}}

We now give a more descriptive presentation of a FIO class whose wave
fronts have generating families with very good properties. In particular,
this class encompasses the FIO's on cotangent lifts
as we will see in details in the next section. 
\begin{assumption*}
Let $(B,p,S)$ be a generating family for a canonical relation from
$\Cot X$ to $\Cot Y$. We denote by $p_{X}$ and $p_{Y}$ the compositions
of the generating family fibration $p:B\rightarrow X\times Y$ with
the canonical projections on the corresponding factors. From now on,
we will assume that:

(1) The manifold $B$ is fibered over its critical submanifold $\pi_{\Sigma}:B\rightarrow\Sigma$,
and any fiber $\pi_{\Sigma}^{-1}(s)$ can be identified with a neighborhood
of $p_{X}(s)$ in $X$. 

(2) The generating function $S:B\rightarrow\R$ has a unique non-degenerate
critical point on the fiber $p_{Y}^{-1}(y)$ for each $y\in Y$.
The set $Z$ of all these critical points is a submanifold of $\Sigma$. 
\end{assumption*}
This assumption implies in particular that a half-density $\Psi$
on $X$ induces a family $\Psi(s)\in|\Omega|^{\frac{1}{2}}(\pi_{\Sigma}^{-1}(s))$
by considering the restrictions of $\Psi$ to neighborhood of $p_{X}(s)$
in $X$. Thus, given another half-density \[
a\in|\Omega|^{\frac{1}{2}}(L)\]
on the canonical relation $L$ generated by $(B,p,S)$, that we transport
with the lagrangian embedding $\lambda$ to the critical submanifold
$\Sigma$ itself, we can identify the tensor product $\Psi\otimes\lambda^{*}a$
with a half-density on $B$. Now, suppose we are given a section $\rho$
of the half-density vertical bundle $|\ker(p_{Y})_{*}|^{\frac{1}{2}}\rightarrow B$.
With this extra data, we can define an operator\[
Q_{\hbar}(a,B):|\Omega|^{\frac{1}{2}}(X)\rightarrow|\Omega|^{\frac{1}{2}}(Y),\]
by fiber integration as explained in paragraph \ref{sub:Averaging-half-densities}:
\begin{equation}
\big(Q_{\hbar}(a,B)\Psi\big)(y):=\int_{p_{Y}^{-1}(y)}\lambda^{*}a\otimes\Psi\otimes\rho\, e^{\frac{i}{\hbar}S}.\label{eq:integral-representation}\end{equation}
This integral representation can be taken as a definition for the
FIO's in $\Four_{\hbar}(L;X,Y)$ provided that the wave front $L$
has a generating family with the properties as above. Note that
any other choice of $\rho$ would yield the same class of FIO. 

This integral representation makes it clear that, once a section
$\rho$ has been fixed, we have a map, called the \textbf{total symbol
map}, \[
\sigma_{B,\rho}:\Four_{\hbar}(L;X,Y)\rightarrow|\Omega|^{\frac{1}{2}}(L),\]
that associates to the operator $T$ its \textbf{total symbol}, that
is, the half-density $a$ on its wave front $L$. 
\begin{rem}
Note that the total symbol is not an invariant of the FIO since
it depends on a choice of a generating family as well as on the choice of
the vertical half-density section $\rho$.
\end{rem}

\subsection{The semiclassical limit}

\subsubsection{The semiclassical intrinsic Hilbert space\label{sub:The-semiclassical-intrinsic}}

We say that a smooth map $\hbar\mapsto f_{\hbar}$ from the parameter
space $(0,1]$ to a normed vector space $V$ is of order $\hbar^{\infty}$
(and we will write $f_{\hbar}\in\mathcal{O}(h^{\infty}))$ if, for
each positive integer $N$, there
is a real positive constant $C_{N}$ such that$ $$\left\Vert f(\hbar)\right\Vert \leq C_{N}\hbar^{N}$. Two paths are equivalent if their difference
is of order $\hbar^{\infty}$. We denote by $V_{\hbar}$ the quotient
space. Whenever $V$ is finite dimensional, the Borel summation Theorem
allows us to identify $V_{\hbar}$ with the space $V[[\hbar]]$
of formal power series in $\hbar$ with coefficients in $V$ by taking
the Taylor series of $f_{\hbar}$ at $0$. In the particular cases when
$V=\C$ or $\R$, we will mostly prefer the interpretation of the
asymptotic spaces $\C_{\hbar}$ and $\R_{\hbar}$ in terms of the
corresponding rings of formal power series $\R[[\hbar]]$ and $\C[[\hbar]]$.

Now let $a_{\hbar}$ and $b_{\hbar}$ be sections of a finite dimensional
vector bundle $E\rightarrow X$ depending smoothly on a parameter
$\hbar\in(0,1]$. We will say that $a_{\hbar}$ and $b_{\hbar}$ are
equivalent modulo $\hbar^{\infty}$ if the difference between their
local representations at any point as well as those of all their derivatives
are of order $\hbar^{\infty}$. In the case of the line bundle of $\alpha$-densities
on $X$, we will write $|\Omega_{\hbar}|^{\alpha}(X)$ for the resulting
quotient space.

The inner product on $|\Omega|^{\frac{1}{2}}(X)$ yields an inner
product on $|\Omega_{\hbar}|^{\frac{1}{2}}(X)$ with values in $\C[[\hbar]]$
in the the sense of \cite{BW2001}.   We will call the resulting inner product space the semiclassical
Hilbert space of the manifold $X$, and we will denote it by $\xH(X)$, though of course
it is not an Hilbert space in the usual sense.
Elements in $\xH(X)$ can be seen as classes of $\hbar$-dependent
half-densities on $X$ that have the same semiclassical limit, that
is, the same $\mathcal{O}(\hbar^{\infty})$ asymptotics when $\hbar\rightarrow0$.
Upon Taylor expansion, we can also regard $\xH(X)$ as the space $|\Omega|^{\frac{1}{2}}(X)[[\hbar]]$
of formal power series in $\hbar$ with coefficients in the half-densities. 

\subsubsection{Oscillatory integrals on microfolds}

Let $\mu_{\hbar}\in|\Omega|^{1}(B)$ be a compactly supported density
on $B$ depending smoothly on $\hbar\in(0,1]$, and let $S:B\rightarrow\R$
be a smooth function whose critical points form a submanifold $Z$
of $B$. The stationary phase theorem tells us that the 
asymptotics modulo $\mathcal{O}(\hbar^{\infty})$ of the oscillatory integral\[
Q_{\hbar}(\mu,S)=\int_{B}\mu_\hbar e^{\frac{i}{\hbar}S},\]
depends only on the behavior of $\mu_\hbar$ and $S$ in a neighborhood of $Z$. To be more precise, we
introduce the following definition:

\begin{defn}
Let $B$ be a manifold and $Z\subset B$ a submanifold. A cut-off
function $\chi:B\rightarrow R$ for $Z$ is a smooth function such
that there exist neighborhoods $U$ and $V$ of $Z$ in $B$ such that
 $\overline U\subset V$ and such
that $\chi_{|U}\equiv1$ and $\chi_{|V^{c}}\equiv0$ where $V^{c}$
is the complement of $V$ in $B$. 
\end{defn}
The stationary phase theorem tells us that $Q_{\hbar}(\mu\chi)$ and
$Q_{\hbar}(\mu\chi')$ are equivalent modulo $\hbar^{\infty}$ for
any two cut-off functions $\chi,\chi':B\rightarrow\R$ for $Z$. For
this reason, any two densities on $X$ having the same germ on $Z$
will have equal oscillatory integrals modulo $\hbar^{\infty}$. This
fact allows us to define semi-classical integrals on manifold germs
or microfolds (see the definition in the Introduction).

\begin{defn}
\label{def:semiclassical-integral}An $\alpha$-density on $[B,Z]$
is an \textbf{$\alpha$-density germ} $[\mu]$ on $B$ around $Z$.
We will use the notation $|\Omega|^{\alpha}([B,Z])$. Let $[\mu]\in|\Omega|^{1}([B,Z])$
be a density germ around $Z$ and let $S:B\rightarrow\R$ be as above.
We define the semiclassical oscillatory integral\[
\SCint\mu e^{\frac{i}{\hbar}S}\]
to be the  equivalence class modulo $\hbar^{\infty}$ of $Q_{\hbar}(\chi\mu)$,
where $\mu\in[\mu]$ and $\chi$ is a cut-off function for $Z$ in
$B$. 
\end{defn}

\subsubsection{Semiclassical Fourier integral operators}

We are now interested in the semiclassical limit of FIO's whose defining
generating families satisfy the assumptions in paragraph \ref{sub:Integral-representation}.
Therefore, instead of considering $Q_{\hbar}(a,B)$ in \eqref{eq:integral-representation}
as a collection of operators indexed by $\hbar\in(0,1]$, we will
see it as single operator\[
Q_{\hbar}(a,B):\xH(X)\longrightarrow\xH(Y)\]
acting on the semiclassical intrinsic Hilbert spaces. Moreover, because
of our assumption that $S$ has a single critical point on each of
the fibers $p_{Y}^{-1}(y)$ and that these critical points form a
submanifold $Z$ of $B$, we see, from the last paragraph, that generating
families $(B,p,S)$ having the same germ around at $Z$ and half-densities
$a$ having the same germ around $\lambda(Z)\subset L$ will yield
operators $Q_{\hbar}(a,B)$ with the same 
asymptotics modulo $\mathcal{O}(\hbar^{\infty})$. In other words, these operators will coincide when looked
upon as acting on the \emph{semiclassical} intrinsic Hilbert spaces.

From the considerations above, it makes sense to introduce germs of
lagrangian submanifolds and generating families to start with. This
motivates the following definition:

\begin{defn}
A generating family for a lagrangian submicrofold $[L,C]$ of a cotangent
bundle $\Cot X$ is a triple $([B,Z],[S],[p])$ for which there is
a representative $(B,S,p)$ that is a generating family for a representative
$L$ of the lagrangian submicrofold and such that

(1) the critical submanifold $\Sigma$ contains $Z$,

(2) the lagrangian embedding $\lambda:\Sigma\rightarrow\Cot X$ maps
 $Z$ diffeomorphically onto $C$.\end{defn}
\begin{example}
As an example of this last definition, let us consider the case of
the conormal bundle $N^{*}C$ seen as the lagrangian submicrofold $[N^{*}C,C]$
of $[\Cot X,X]$.\footnote{From now on, we will write simply $\Cot X$ for $[\Cot X,X]$ when
it is clear from the context that we are interested in the microfold%
} The generating family $(B_{C}^{\Psi},p_{C},S_{C})$ of Example \ref{sub:Conormal-microbundle-deformations}
induces a generating family for the microbundle $N^{*}C$ by taking
germs appropriately. Namely, one easily sees that the information
needed to generate the lagrangian embedding germ\begin{eqnarray*}
[\iota_{C}]:N^{*}C & \longrightarrow & [\Cot X,C]\end{eqnarray*}
is contained in the microfold $[B_{C},Z_{C}]$, where $Z_{C}$ is
the submanifold of points of the form $(0,c,c)$ with $c\in C$, along with the germs\begin{eqnarray*}
[p_{C}]:[B_{C},Z_{C}] & \longrightarrow & [\Cot X,C],\\
{}[S_{C}]:[B_{C},Z_{C}] & \longrightarrow & [\R,0].\end{eqnarray*}
Since $Z_{C}$ is contained in $\Sigma_{C},$ we can take the germ
$[\Sigma_{C},Z_{C}]$, which we will call the {\bf critical submicrofold}
of the generating function germ $[S_{C}]$. As before, we have the
embedding\begin{eqnarray*}
[\tau_{C}]:N_{*}C & \longrightarrow & [B_{C},Z_{C}],\end{eqnarray*}
whose image is $[\Sigma_{C},Z_{C}]$. This data produces the lagrangian
embedding germs\begin{eqnarray*}
[\lambda_{C}]:[\Sigma_{C},Z_{C}] & \longrightarrow & [\Cot X,C],\\
{}[\iota_{C}]:N^{*}C & \longrightarrow & [\Cot X,C],\end{eqnarray*}
where $[\iota_{C}]=[\lambda_{C}\circ\tau_{C}]$ as before. 
\end{example}
Forgetting now everything but the $\mathcal{O}(\hbar^{\infty})$ asymptotics
of our FIO's, we can directly start from the data of a generating
family $([B,Z],[p],[S])$ of a lagrangian submicrofold $[L,C]$ in $\Cot(X\times Y)$
for which a representative satisfies the assumptions in paragraph
\eqref{sub:Integral-representation} and half-density germs $[a]$
on the lagrangian submicrofold $[L,C]$.  In this way, we obtain operators
$Q_{\hbar}([a],[B,Z])$ from $\xH(X)$ to $\xH(Y)$ by replacing the
integral in \eqref{eq:integral-representation} by its semiclassical
version introduced in Definition \ref{def:semiclassical-integral}.
We call these operators the {\bf semiclassical Fourier integral operators},
and we denote their space by $\SFour(X,Y)$. Note that the wave front map $\WF$
now takes its values in lagrangian submicrofolds and the total
symbol map in half-density germs.

\section{Quantization of cotangent lifts\label{sec:Quantization-of-cotangent}}

In this section, we study the semiclassical FIO's associated
to the category $\mathcal{C}$ of cotangent lifts. For this, we need to introduce the notion
of micro (and local) exponential:

\begin{defn}
\label{def:exponential}Let $M$ be a smooth manifold. A \textbf{micro}
\textbf{exponential} is a diffeomorphism germ $[\Psi]:[TM,Z_{M}]\rightarrow M$
that sends  the zero section  
$Z_{M}$ to $M$ and such
that, for each $x$, $T_{0}\Psi_{x}=\id$, where $\Psi_{x}$ is the restriction of
$\Psi$ to $T_{x}M$. A \textbf{local} \textbf{exponential} is a representative
$\Psi\in[\Psi]$ of a micro exponential. We will usually denote the
domain of $\Psi_{x}$ by $U_{x}$ and its range by $V_{x}$. 
\end{defn}

We next show that cotangent lifts are very special in the following
sense:

\begin{prop}
\label{prop-properties}
Let $\Cot\phi:\Cot M\rightarrow\Cot N$ be a cotangent lift to a smooth map $\phi$ from $N$ to $M$ and let
$\Psi:U\subset TM\rightarrow M$ be a local exponential on $M$ as
in Definition \ref{def:exponential}. Then, there is a canonical generating
family $(B_{\phi}^{\Psi},p_{\phi},S_{\phi})$, depending only on the
choice of the local exponential, such that

(1) the assumptions in paragraph \ref{sub:Integral-representation}
hold;

(2) the image of the critical points $Z_{\phi}$ of $S_{\phi}$ on
the $p_{N}$-fibers by the lagrangian embedding is the graph of $\phi$
seen as a submanifold of $\Cot\phi$;

(3) there is a canonical section $\rho_{B}:N\rightarrow|\ker(p_{N})_{*}|^{\frac{1}{2}}$ 
\end{prop}
Let us see some immediate consequences of these good properties.

First of all,  the  $\mathcal{O}(\hbar^{\infty})$ asymptotics of  FIO's from $\xH(M)$ to $\xH(N)$ whose wave front are the cotangent lifts and with integral representation \begin{equation}
\Big(Q_{\hbar}(a,\Cot\phi)u\Big)(x_{2}):=\SCint_{p_{N}^{-1}(x_{2})}u\otimes a\otimes\rho_{B}\, e^{\frac{i}{\hbar}S_{\phi}},\label{eq:QuantFormula}
\end{equation}
where $u\in\xH(M)$ are completely determined by (1) the lagrangian submicrofolds $[\Cot\phi,\graph\phi]$, (2) the corresponding generating family germ
\[
\Big([B_{\phi}^{\Psi},Z_{\phi}],[p_{\phi}],[S_{\phi}]\Big),
\]
and (3) the half-density germs $[a]\in|\Omega_{\hbar}|^{\frac{1}{2}}([\Cot\phi,\graph\phi])$. Moreover, the whole calculus depends only on
the germs of the local exponentials around the zero sections, that
is, the micro exponentials\[
[\Psi]:[\Cot M,Z_{M}]\rightarrow M\]
of Definition \ref{def:exponential}.

Second of all, the class of semiclassical FIO's on cotangent lifts
has very good functorial properties. Since the cotangent lifts form
a category $\mathcal{C}$, their associated semiclassical
FIO's are always composable when sources and targets are compatible, and so the collection $\SFour_{\hbar}^{\mathcal{C}}$ of these FIO's also forms a category, and we have
a wave front functor $$
\WF:\SFour_{\hbar}^{\mathcal{C}}\longrightarrow\mathcal{C},$$
as explained in the previous section. 

The rest of this section is devoted to the proof of Proposition \ref{prop-properties}.
Throughout, $M$ and $N$ will be two smooth manifolds as above whose
points will be denoted  by $x_{1}$ and $x_{2}$ respectively. Correspondingly,
we will write $(p_{1},x_{1})$ and $(v_{1},x_{1})$ to denote points
in $\Cot M$ and $TM$ and $(p_{2},x_{2})$ and $(v_{2},x_{2})$
for points in $\Cot N$ and $TN$. 

\subsection{Canonical identifications }

Let $\phi:N\rightarrow M$ be a smooth map. The conormal bundle of
its graph $N^{*}(\graph\phi)$ can be identified with the pullback
bundle $\phi^{*}(\Cot M)$ via the the lagrangian embedding\begin{eqnarray*}
\iota_{\phi}:\phi^{*}(\Cot M) & \longrightarrow & \Cot(M\times N),\end{eqnarray*}
given by\begin{eqnarray}
\iota_{\phi}(p_{1},x_{2}) & = & \Big(\big(p_{1},-d\phi^{*}p_{1}\big),\big(\phi(x_{2}),x_{2}\big)\Big).\label{eq:cotangent-lift-inclusion}\end{eqnarray}
Let us list here several obvious but crucial identifications, which
will be used extensively in what follows:\begin{eqnarray}
\graph\phi & \simeq & N,\label{eq:graphphi}\\
N^{*}\graph\phi & \simeq & \phi^{*}(\Cot M),\label{eq:Nstargraphphi}\\
N\graph\phi & \simeq & \phi^{*}(TM),\label{eq:Ngraphphi}\end{eqnarray}
and, therefore, the vector bundle $N^{*}\graph\phi\oplus N\graph\phi$
over $\graph\phi$ can be identified with the pullback vector bundle
$\phi^{*}(\Cot M\oplus TM)$ over $N$. From now on, we will not make
a strict distinction between $N^{*}\graph\phi$ and the cotangent
lift $\Cot\phi$, and similarly for the other identifications.

\subsection{Generating families, tubular neighborhoods and micro exponentials\label{sub:MorseFamForCot}}

Example \ref{sub:ConormBundle} shows how to construct a generating
family for a general conormal bundle once a tubular neighborhood for
the submanifold has been given. Since cotangent lifts are essentially
conormal bundles, this construction works also for them. However,
we want to discuss a special class of tubular neighborhoods in the
context of cotangent lifts, which will prove to be very convenient
for us. 

\begin{rem}
A connection $\nabla$ on $M$ gives rise to a micro exponential by
taking the germ of the connection's exponential map $\exp^{\nabla}$. If
we replace germs by jets in the previous definition, we obtain what
is called a \textbf{formal} \textbf{exponential} in \cite{WE93} in
the context of Fedosov star-products. A micro exponential produces
a {}``micro linearization'' of the manifold $M$, that is, a germ\begin{eqnarray*}
[l]:[M\times M,\Delta_{M}] & \longrightarrow & [\Cot M,Z_{M}],\end{eqnarray*}
given in representatives by $l(x,y)=v$, where $\Psi_{x}(v)=y$. The
notion of manifold linearization has been introduced in \cite{Bokobza69}
in order to extend the pseudo-differential operator calculus on
$\R^{n}$ to general manifolds. Micro exponentials are also related
to Milnor's construction of tangent microbundles \cite{Milnor1964}.
\end{rem}
Returning to cotangent lifts, we can now easily construct a tubular
neighborhood for the graph of $\phi:N\rightarrow M$ out of the extra
data of a micro exponential $[\Psi]$ on $M$ only. For the neighborhood
itself, we take\[
V_{\phi}:=\bigcup_{x_{2}\in N}\Psi(U_{\phi(x_{2})})\times\phi^{-1}(\Psi(U_{\phi(x_{2})})),\]
where $U_{x_{1}}$ is the domain of $\Psi_{x_{1}}$ for a fixed representative
$\Psi\in[\Psi]$. Now, using the identification of the graph normal
bundle with the pull back $\phi^{*}(TM)$, we obtain the tubular
neighborhood diffeomorphism germ from $N\graph\phi$ to $V_{\phi}$
given explicitly by\begin{eqnarray*}
(v_{1},x_{2}) & \longmapsto & \big(\Psi_{\phi(x_{2})}(v_{1}),x_{2}\big).\end{eqnarray*}
We will denote this map again by $\Psi$ in order to keep the notation
simple and to acknowledge the micro exponential dependence in the
tubular neighborhood notation $(V_{\phi},\Psi)$.

Now, repeating the generating family construction for conormal bundles
given in Example \ref{sub:Conormal-microbundle-deformations} with
the tubular neighborhood $(V_{\phi},\Psi)$, we obtain a generating
family $(B_{\phi}^{\Psi},p_{\phi},S_{\phi})$ for the cotangent lift
$\Cot\phi$. In very explicit terms, we have that\begin{eqnarray*}
B_{\phi}^{\Psi} & = & \Big\{(p_{1},x_{1},x_{2})\,:p_{1}\in\Cot_{\phi(x_{2})}M,\, x_{1}\in\im\Psi_{\phi(x_{2})},\, x_{2}\in N\Big\},\\
p_{\phi}(p_{1},x_{1},x_{2}) & = & (x_{1},x_{2}),\\
S_{\phi}(p_{1},x_{1},x_{2}) & = & \big\langle p_{1},\Psi_{\phi(x_{2})}^{-1}(x_{1})\big\rangle,\\
\Sigma_{\phi} & = & \Big\{(p_{1},\phi(x_{2}),x_{2})\,:p_{1}\in\Cot_{\phi(x_{2})}M,\, x_{2}\in N\Big\}.\end{eqnarray*}
Again, we denote by $\tau_{\phi}$ the embedding of $\phi^{*}(\Cot M)$
into $B_{\phi}^{\Psi}$ and by $\lambda_{\phi}$ the lagrangian embedding
of $\Sigma_{\phi}$ into $\Cot(M\times N)$. The composition $\iota_{\phi}:=\lambda_{\phi}\circ\tau_{\phi}$
yields the usual inclusion \eqref{eq:cotangent-lift-inclusion}. 

Let us check now that this generating family satisfies the assumptions
of paragraph \ref{sub:Integral-representation}. First of all, observe
that we have a retraction $\pi_{\Sigma}$ from $B_{\phi}^{\Psi}$
onto the critical submanifold $\Sigma$ given by\[
\pi_{\Sigma}:(p_{1},x_{1},x_{2})\mapsto(p_{1},\phi(x_{2}),x_{2}),\]
the fiber of which is the neighborhood $V_{x_{1}}$ of $x_{1}$
in $M$ determined by the exponential $\Psi$. This shows the first
part of the assumption. Now, if we consider the restriction of $S_{\phi}$
on the fiber\[
p_{N}^{-1}(x_{2})=\Cot_{\phi(x_{2})}M\times\Psi_{\phi(x_{2})}(U_{\phi(x_{2})}),\]
and if compute its critical point there, we see that the vanishing
of\[
\frac{\partial S_{\phi}}{\partial p_{1}}(p_{1},x_{1},x_{2})=\Psi_{\phi(x_{2})}^{-1}(x_{1})\]
on the fiber implies that $\Psi_{\phi(x_{2})}^{-1}(x_{1})=0,$ which
is equivalent to $x_{1}=\phi(x_{2})$. In turn, the vanishing of\[
\frac{\partial S_{\phi}}{\partial x_{1}}(p_{1},x_{1},x_{2})=\langle p_{1},T_{x_{1}}\Psi_{\phi(x_{2})}^{-1}\rangle\]
at $x_{1}=\phi(x_{2})$ implies that $p_{1}=0$ since the derivative
$T_{\phi(x_{2})}\Psi_{\phi(x_{2})}^{-1}$ is the identity by definition
of the exponential. This shows that $S_{\phi}$ has a single critical
point on $p_{N}^{-1}(x_{2})$ given by $(0,\phi(x_{2}),x_{2})$. We
see that the collection of all these critical points is the following
submanifold \[
Z_{\phi}:=\Big\{(0,\phi(x_{2}),x_{2}):\; x_{2}\in N\Big\},\]
which is contained in the (usual) critical submanifold $\Sigma_{\phi}$
of the generating family. Finally, we see that the lagrangian embedding
$\lambda_{\phi}$ carries $Z_{\phi}$ to the graph of $\phi$, seen
as a submanifold of the cotangent lift $\Cot\phi$. 

We see then that the semiclassical limit of an FIO whose wave front is a cotangent
lift is controlled by the lagrangian microfold $[\Cot\phi,\graph\phi]$
(which we will continue to denote  simply by $\Cot\phi$)
and its generating family germ \begin{equation}
\Big([B_{\phi}^{\Psi},Z_{\phi}],[p_{\phi}],[S_{\phi}]\Big).\label{eq:cotangent-lift-microfamily}\end{equation}
\subsection{The canonical half-density $\rho_{B}$}

The fibration $p_{N}:B_{\phi}^{\Psi}\rightarrow N$  has a section, given by $s_{N}(x_{2})=(0,\phi(x_{2}),x_{2})$, whose image is $Z_{\phi}$. We will exhibit a canonical section $\rho_{B}$
of the half-density bundle\begin{eqnarray}
|\ker(p_{N*})|^{\frac{1}{2}} & \longrightarrow & B_{\phi}^{\Psi}\label{eq:can-dens-bundle}\end{eqnarray}
associated to the vertical bundle of this projection. 

Let us introduce the shorthand $\T M$ for the vector bundle $\Cot M\oplus TM$
over $M$. As a first step, observe that the half-density bundle $|\T M|^{\frac{1}{2}}$
has a canonical section $\rho_{M}$ obtained fiberwise from the canonical
symplectic form on $\T_{x}M$, seen as a symplectic vector space for
each $x\in M$. This section $\rho_{M}$ produces a new section, which
we will continue to denote by $\rho_{M}$, on the pullback bundle\begin{eqnarray*}
(\phi\circ p_{N})^{*}(\T M) & \longrightarrow & B_{\phi}^{\Psi}.\end{eqnarray*}
Our strategy now is to construct a vector bundle isomorphism $A$
between $(\phi\circ p_{N})^{*}(\T M)$ and $\ker(p_{N})_{*}$ and
to use it in order to carry $\rho_{M}$ to a section of \eqref{eq:can-dens-bundle},
which will be our canonical half-density $\rho_{B}$. In order to
produce $A$, we can first observe that the tangent space of $B_{\phi}^{\psi}$
at $z=(p_{1},x_{1},x_{2})$ can be decomposed into the direct sum\begin{eqnarray*}
T_{z}B_{\phi}^{\Psi} & = & \Cot_{\phi(x_{2})}M\oplus T_{x_{1}}M\oplus T_{x_{2}}N.\end{eqnarray*}
Thus, the vertical bundle fiber at $z$ is\begin{eqnarray*}
\ker_{z}(p_{N*}) & = & \Cot_{\phi(x_{2})} M\oplus T_{x_{1}}M\end{eqnarray*}
since it is the tangent space to the fiber $p_{N}^{-1}(x_{2})$ at
$z$. At this point, we can write down the desired vector bundle isomorphism
$A$ by using the derivative of a micro exponential. Fiberwise, $A(z)$
is the linear map from $(\phi\circ p_{N})^{*}(\T M)_z=T^*_{\phi(x_2)}M\oplus T_{\phi(x_2)}M$ to $\ker_{z}(p_{N*})=\Cot_{\phi(x_{2})} M\oplus T_{x_{1}}M$ 
given in matrix form by 
\begin{eqnarray*}
A(z) & = & \left(\begin{array}{cc}
\id_{\Cot_{\phi(x_2)} M} & 0\\
0 & \partial_{v_{1}}\Psi_{\phi(x_{2})}\big(\Psi_{\phi(x_{2})}^{-1}(x_{1})\big)\end{array}\right).
\end{eqnarray*}
Finally, we can set\begin{eqnarray*}
\rho_{B}(z)(\mathbf{e}) & :=\frac{1}{(2\pi\hbar)^{\frac{m+n}{2}}} & \rho_{M}(z)(\mathbf{e}\cdot A(z)^{-1}),\end{eqnarray*}
 where $m=\dim M$ and $n=\dim N$. 
\begin{example}
\label{exa:rhoB}Let us compute $\rho_{B}$ in the case where $M=\R^{m}$, $N=\R^{n}$ and $\Psi_{x}:\R^{m}\rightarrow\R^{m}$ is a global
diffeomorphism of for all $x\in\R^{n}$. In this case, we have that\begin{eqnarray*}
B_{\phi}^{\Psi} & = & \R_{p_{1}}^{m}\times\R_{x_{1}}^{m}\times\R_{x_{2}}^{n},\end{eqnarray*}
and we can identify the fibers of both $(\phi\circ p_{N})^{*}(\T M)$
and $\ker(p_{N})_{*}$ with the vector space $V:=(\R^{m})^{*}\oplus\R^{m}$.
The canonical section $\rho_{M}$ is the constant symplectic half
density\begin{eqnarray*}
\rho_{M}(x_{2}) & = & \sqrt{dp_{1}}\sqrt{dx_{1}}\end{eqnarray*}
on the symplectic vector space $V$. Hence, we obtain that\begin{eqnarray*}
\rho_{B}\big((p_{1},x_{1},x_{2})\big) & = & \Big|\det\partial_{v_{1}}\Psi_{\phi(x_{2})}\big(\Psi_{\phi(x_{2})}^{-1}(x_{1})\big)\Big|^{-\frac{1}{2}}\frac{\sqrt{dp_{1}}\sqrt{dx_{1}}}{(2\pi\hbar)^{m}},\end{eqnarray*}
since $\rho_{M}(x_{2})(\mathbf{e}\cdot A(z)^{-1})=|\det A(z)^{-1}|^{\frac{1}{2}}\rho_{M}(x_{2})$.

In particular, in the case of the global exponential coming from the
canonical affine connection on $\R^{m}$, $\Psi_{x_{1}}(v_{1})=x_{1}+v_{1}$,
we have that \begin{eqnarray*}
\big|\det\partial_{v_{1}}\Psi_{\phi(x_{2})}(v_{1})\big|^{\frac{1}{2}} & = & 1\end{eqnarray*}
in the above formula for $\rho_{B}(z)$. 
\end{example}

\subsection{The Fourier integral operator $Q_{h}(a,\Cot\phi)$}

In this paragraph, we put everything together to construct the
semi-classical Fourier integral operator\begin{eqnarray*}
Q_{\hbar}(a,\Cot\phi):\xH(M) & \longrightarrow & \xH(N)\end{eqnarray*}
associated to a half-density $a\in|\Omega_{\hbar}|^{\frac{1}{2}}(\phi^{*}(\Cot M))$.
It goes as follows: we start with $u\in\xH(M)$ . Now, the restriction
of the half-density product\[
u\otimes a\in|\Omega_{\hbar}|^{\frac{1}{2}}\big(M\times\phi^{*}(\Cot M)\big)\]
to the submanifold $B_{\phi}^{\Psi}\subset(M\times\phi^{*}(\Cot M)\big)$
produces a half-density on $B_{\phi}^{\psi}$, which we will still
write as $u\otimes a$. We are now in the case of Section \ref{sub:Averaging-half-densities},
with a half-density $u\otimes a\otimes\, e^{\frac{i}{\hbar}S_{\phi}}$
on $B_{\phi}^{\Psi}$ and the canonical half-density section $\rho_{B}$
on the fibers of the projection $p_{N}:B_{\phi}^{\Psi}\rightarrow N$.
This allows us to define our operator pointwise as the semi-classical
oscillatory integral \eqref{eq:QuantFormula} since, as already shown,
$S_{\phi}$ has a unique critical point in the fiber $p_{N}^{-1}(x_{2})$,
namely at $s_{N}(x_{2})=(0,\phi(x_{2}),x_{2})$.
\begin{example}
\label{exa:Operator}We continue Example \ref{exa:rhoB}, in which
 $M=\R^{m}$ , $N=\R^{n}$, the micro exponential is a global
diffeomorphism  (as for instance the one coming from the canonical affine connection, i.e., $\Psi_{x_{1}}(v_{1})=x_{1}+v_{1}$). In coordinates, we have \begin{eqnarray*}
S_{\phi}(p_{1},x_{1},x_{2}) & := & \big\langle p_{1},\Psi_{\phi(x_{2})}^{-1}(x_{1})\big\rangle,\\
a & := & a(p_{1},x_{2})\sqrt{dp_{1}}\sqrt{dx_{2},}\\
u & := & u(x_{1})\sqrt{dx_{1}}.\end{eqnarray*}
Using the formula for $\rho_{B}$ computed in Example \ref{exa:rhoB},
we obtain that
\begin{eqnarray*}
au\rho_{B} & = & \big(a(p_{1},x_{2})u(x_{1})\Big|\det\partial_{v_{1}}\Psi_{\phi(x_{2})}\big(\Psi_{\phi(x_{2})}^{-1}(x_{1})\big)\Big|^{-\frac{1}{2}}\sqrt{dx_{2}}\big)\,\frac{dp_{1}dx_{1}}{(2\pi\hbar)^{ym}},\end{eqnarray*}
is a density on $p_{N}^{-1}(x_{2})$ with values in $|\Cot_{x_{2}}N|^{\frac{1}{2}}$.
Hence, $\big(Q_{\hbar}(a,T^{*}\phi)u\big)(x_{2})$ coincides with
\begin{equation*}
\left(\SCint a(p_{1},x_{2})u(x_{1})\Big|
\det\partial_{v_{1}}\Psi_{\phi(x_{2})}\Big|^{-\frac{1}{2}}e^{\frac{i}{\hbar}\langle p_{1},\Psi_{\phi(x_{2})}^{-1}(x_{1})\rangle}\frac{dp_{1}dx_{1}}{(2\pi\hbar)^{m}}\right)\sqrt{dx_{2}}.
\end{equation*}
In the case, we take the canonical affine connection
$\Psi_{x_{1}}(v_{1})=x_{1}+v_{1}$,
we obtain the semi-classical version of pseudo-differential operators.
\end{example}

\section{Quantization of symplectic micromorphisms\label{sec:Quantization-of-symplectic}}

In this section, we extend the semiclassical FIO calculus on cotangent
lifts developed in the previous section to a more general class of
wave fronts. We do so by deforming the canonical generating families
of cotangent lifts. It turns out that these new wave fronts belong
to a special class of canonical relation germs: the class of symplectic
micromorphisms which are the morphisms of the cotangent microbundle
category $\Mic$ as constructed in \cite{SM_Microfolds}. As is the case for
cotangent lifts, symplectic micromorphisms always compose well,
and we obtain a new category of semiclassical Fourier integral operators
together with a wave front functor\[
\WF:\SFour^{\Mic}\rightarrow\Mic.\]
This new category encompasses well known examples of semiclassical FIO's,
the main example of which is the class of pseudo-differential operators. 

Since our deformations of cotangent lift generating families have
a meaning for general conormal bundles, we start with this case.

\subsection{Conormal microbundle deformations\label{sub:Conormal-microbundle-deformations}}

Our goal here is to describe a class of lagrangian deformations of
conormal bundles along a tubular neighborhood and show how to obtain
their generating families as deformations of the standard
generating family of Example \ref{sub:ConormBundle}. 

A tubular neighborhood $(V,\Psi)$ of $C$ in $X$ fibers the neighborhood
$V$ of $C$ into slices $V_{c}:=\Psi(U_{c})$ over the points $c\in C$.
The conormal bundle $N^{*}V_{c}$ of each of these slices is a lagrangian
submanifold of $\Cot X$, which is transversal to the conormal bundle
of $C$. Moreover, $N^{*}C$ and $N^{*}V_{c}$ intersect only at $c$.
The distribution $\Lambda$ over $C$ given by the collection of tangent
spaces to $N^{*}V_{c}$ for each $c\in C$ is thus a lagrangian distribution
over $C$, which is transversal to $N^{*}C$, hence we have the
lagrangian splitting\[
T(\Cot X)=T(N^{*}C)\oplus\Lambda\]
of the tangent space to $\Cot X$ along $C$.
\begin{defn}
\label{def:conormal-deformation}A \textbf{deformation} of the lagrangian
microfold $N^{*}C$ along a tubular neighborhood $(V,\Psi)$ of $C$
in $M$ is a lagrangian submicrofold $[L,C]$ of $\Cot X$ that is
transversal to the lagrangian distribution $\Lambda$ defined above.
\end{defn}
Our goal is to show that all deformations of $N^{*}C$ along a given
tubular neighborhood are obtained from the standard conormal bundle
generating family by deforming its generating function of the conormal bundle as follows:
\begin{defn} Let $[S_C]$ be the generating function of the conormal bundle to $C$. 
We call a deformation of $[S_{C}]$ a function germ of the form
$
	[S_{C}^{f}] =  [S_{C}]-[r_{C}^{*}f].
$
where \[
[f]\in C_{0}^{\infty}(N^{*}C,C)\quad\textrm{s.t.}\quad\partial_{p}f(0,x)=0.
\]
\end{defn}

\begin{prop}
\label{pro:deformation}The triple $\big([B_{C}^{\Psi},Z_{C}],[p_{C}],[S_{C}^{f}]\big)$
generates a deformation $[L_{f}^{\Psi},C]$ of $N^{*}C$ along the
tubular neighborhood $(V,\Psi)$. Conversely, all such deformations
arise this way. Moreover, the critical submicrofold $[\Sigma_{C}^{f},Z_{C}]$
is given, in representatives, by all the points $z\in B_{C}^{\Psi}$
of the form\[
z=\big(p,\Psi_{c}(\partial_{p}f(p,c),c\big),\]
where $(p,c)$ is taken in an appropriate neighborhood of the zero
section of $N^{*}C$. The corresponding lagrangian embedding germ
$[\lambda_{C}^{f}]:[\Sigma_{C}^{f},Z_{C}]\rightarrow[\Cot X,C]$ is
given explicitly by\[
\lambda_{C}^{f}(z)=\Big(\Cot_{v_{f}}\Psi_{c}(p),\Psi_{c}(c)\Big),\]
where $v_{f}=\partial_{p}f(p,c)$. \end{prop}
\begin{rem}
A consequence of Proposition \ref{pro:deformation} is that the critical
submicrofolds for all deformations arise as follows: for each deformation
$[f]$, there is a diffeomorphism germ\begin{eqnarray*}
[\chi_{f}]:[\Sigma_{C},Z_{C}] & \longrightarrow & [\Sigma_{C}^{f},Z_{C}],\end{eqnarray*}
that fixes the core. It is given in representatives by\begin{eqnarray*}
\chi_{f}(p,c,c) & = & \Big(p,c,\Psi_{c}(\partial_{p}f(p,c))\Big).\end{eqnarray*}
Observe also that we obtain a family of lagrangian embeddings $[\iota_{C}^{f}]:N^{*}C\rightarrow[\Cot X,C]$
deforming the canonical inclusion by composition: $[i_{C}^{f}]=[\lambda_{C}^{f}\circ\chi^{f}\circ\tau_{C}]$.
Explicitly, we have\[
\iota_{C}^{f}(p,c)=\Big(\Cot_{v_{f}}\Psi_{c}(p),\Psi_{c}(c)\Big).\]
\end{rem}
\begin{proof}
Let us first prove that the triple $\big([B_{C}^{\Psi},Z_{C}],[p_{C}],[S_{C}^{f}]\big)$
is a generating family. This amounts to showing that there is a representative
$S_{C}^{f}\in[S_{C}^{f}]$ that is non-degenerate. In is enough to
work locally. In local coordinates, we have that\begin{eqnarray*}
S_{C}^{f}(p,x,c) & = & \langle p,\Psi_{c}^{-1}(x)\rangle-f(p,c),\end{eqnarray*}
and that the critical set $\Sigma_{C}^{f}$ is the locus of points
such that $H(p,x,c)=0$ with\begin{eqnarray*}
H(p,x,c) & := & \Psi_{c}^{-1}(x)-\partial_{p}f(p,c).\end{eqnarray*}
Observe now that, for all $(0,c,c)\in Z_{C}$, we have $H(0,c,c)=0$
and \begin{eqnarray*}
\frac{\partial H}{\partial x}(0,c,c) & = & \id.\end{eqnarray*}
Hence, an application of the implicit function theorem tells us that,
for each $(p,c)$ with $p$ small enough, there is an unique $x(p,c)$
such that $H(p,x(p,c),c)=0$, and thus the solutions of this equation
form a submanifold $\Sigma_{C}^{f}$ containing $Z_{C}$. Taking the
germ $[\Sigma_{C}^{f},Z_{C}]$ yields thus a submicrofold: the critical
submicrofold of $[S_{C}^{f}]$. Actually, the equation $H(p,x,c)=0$
is explicitly solvable since $f$ is independent of $x:$ for each
$(p,c)$ sufficiently close to the zero section, we can take
\[x(p,c)=\Psi_{c}(\partial_{p}f(p,c)),\]
which gives the form of the critical submicrofold in representatives.
It follows that $\big([B_{C}^{\Psi},Z_{C}],[p_{C}],[S_{C}^{f}]\big)$ is a
generating family generating a lagrangian submicrofold, which we denote
by $[L_{f}^{\Psi},C]$. A straightforward computation now gives the
form of the lagrangian embedding germ $\lambda_{c}^{f}$.

It remains to see that $[L_{f}^{\Psi},C]$ is a deformation of $N^{*}C$
and that all deformations arise this way. For this, let us identify
the cotangent bundle $\Cot N^{*}C$ with $N^{*}C\oplus NC$ and introduce the diffeomorphism germ 
\[g:[\Cot N^{*}C,C]\longrightarrow[B_{C}^{\Psi},Z_{C}],\]
that maps $(p,v,c)$ to $(p,\Psi(v),c)$. It produces an equivalent
generating family\[
\big([\Cot N^{*}C,C],\, p_{C}\circ g\,,S_{C}\circ g\big)\]
whose critical submicrofold is the conormal microbundle $N^{*}C$.
The lagrangian embedding germ of this new generating family can be
conveniently described as the restriction to the critical submicrofold
of a symplectomorphism germ\[
[\chi]:[\Cot N^{*}C,C]\longrightarrow[\Cot X,C],\]
that sends the vertical distribution $V(\Cot N^{*}C)$ over $C$ to
the lagrangian distribution $\Lambda$.  (The existence of such a germ $[\chi]$
is guaranteed by Theorem 7.1 in \cite{Weinstein1971}.)
Now, in this new
description, the generating function deformations read\[
(S_{C}^{f}\circ g)(p,v,c)=\langle p,v\rangle-f(p,c).\]
The corresponding critical submicrofolds are nothing but the lagrangian
submicrofolds $[\graph df,C]$ and the lagrangian embedding germs
are given by the restriction of $[\chi]$ to them. Since lagrangian
submicrofolds of the form $[\graph df,C]$ are exactly the lagrangian
submicrofolds through $C$ that are transversal to the vertical distribution
in $\Cot N^{*}C$, then the generated lagrangian submicrofolds $[\chi(\graph df),C]$
are exactly all the lagrangian submicrofolds through $C$ that are
transversal to $\Lambda$, that is, all the deformations of $N^{*}C$
along the tubular neighborhood. 
\end{proof}

\subsection{Symplectic micromorphisms as deformed cotangent lifts}

In micro-geometry,  we can define a canonical relation from $\Cot M$ to $\Cot N$
to be a lagrangian submicrofold 
\[
[V,\graph\phi]\subset\overline{\Cot M}\times\Cot N\]
whose core is the graph of a smooth map $\phi:N\rightarrow M$. Although
canonical relations do not compose well in general, it is possible
to single out a class for which they do. This class is the class of
symplectic micromorphisms as introduced in \cite{SM_Microfolds}.
To distinguish them, we will use the special notation\begin{eqnarray*}
([V],\phi):\Cot M & \longrightarrow & \Cot N,\end{eqnarray*}
instead of $[V,\graph\phi]$. They can be characterized in several
ways, the most useful for us now being the following: 
\begin{defn}
A symplectic micromorphism $([V],\phi)$ from $\Cot M$ to $\Cot N$
is a canonical relation $[V,\graph\phi]$ that is transversal to the
lagrangian distribution\begin{eqnarray}
\Lambda: & = & TZ_{M}\oplus V(\Cot N).\label{eq:lagr-distrib}\end{eqnarray}
over $\graph\phi$. 
\end{defn}
Cotangent lifts are symplectic micromorphisms, but not all symplectic
micromorphism arise this way. However, any symplectic micromorphism
can be realized as a cotangent lift deformation in the sense of paragraph
\ref{sub:Conormal-microbundle-deformations}. More precisely, given
a micro exponential $[\Psi]$ on $M$, one can deform the associated
generating family\[
\big([B_{\phi}^{\Psi},Z_{\phi}],\,[p_{\phi}],\,[S_{\phi}]\big)\]
of the cotangent lift $\Cot\phi:\Cot M\rightarrow\Cot N$ by adding
a function germ to it:\begin{equation}
S_{\phi}^{f}(p_{1},x_{1},x_{2})=\langle p_{1},\Psi_{\phi(x_{2})}^{-1}(x_{1})\rangle-f(p_{1},x_{2}),\label{eq:deformed-generating-function}\end{equation}
where $[f]$ is a function germ on the conormal bundle $N^{*}(\graph\phi)$
(identified, as usual, with the pullback bundle $\phi^{*}(\Cot M)$)
that is identically null on the zero section and whose derivative
in the fiber direction vanishes. This yields a lagrangian submicrofold
$[L_{\Psi}^{f},\graph\phi]$, which is a deformation of $\Cot\phi$
along the tubular neighborhood $(V_{\phi},\Psi)$ in the sense of
Definition \ref{def:conormal-deformation}. 

Now, as explained in paragraph \ref{sub:Conormal-microbundle-deformations},
$[L_{\Psi}^{f},\graph\phi]$ is transversal to the lagrangian distribution
$\Lambda^{\Psi}$ given by tangent spaces at points of $\graph\phi$
to the conormal bundles $N^{*}V_{(\phi(x),x)}$ of the tubular neighborhood
slices\[
V_{(\phi(x),x)}:=\Psi_{\phi(x)}(U_{\phi(x)})\times\{x\}.\]
As is easily checked, the lagrangian distribution $\Lambda^{\Psi}$
is independent of the micro exponential and coincides with the lagrangian
distribution \eqref{eq:lagr-distrib} defining symplectic micromorphisms.
Therefore, an application of Proposition \ref{pro:deformation} to
this case immediately yields:
\begin{prop}
\label{pro: generating functions} Once a micro exponential
$[\Psi]$ on $M$ is fixed, there is a one-to-one correspondence between the
symplectic micromorphisms from $\Cot M$ to $\Cot N$ with core map
$\phi$ and the deformations $[L_{\Psi}^{f},\graph\phi]$ of the core
map cotangent lift along the tubular neighborhood $(V_{\phi},\Psi)$. 
\end{prop}

\subsection{Enhancements and quantization}

Our goal here is to realize the class of semiclassical FIO's on symplectic
micromorphisms in terms of a two-step construction performed on them:
enhancement and quantization. 
\begin{defn}
An \textbf{enhancement} of a symplectic micromorphism $\big([V],\phi\big):\Cot M\rightarrow\Cot N$
is a half-density germ\[
[a]\in|\Omega|^{\frac{1}{2}}\big([V,\graph\phi]\big).\]
The triple $([a],[V],\phi):\Cot M\rightarrow\Cot N$ will be called
an \textbf{enhanced symplectic micromorphism}.
\end{defn}
Let us fix a micro exponential $[\Psi]$ on $M$ and let $([B_{\phi}^{\Psi},Z_{\phi}],p_{\phi},S_{\phi}^{f})$
be the corresponding generating family of $([V],\phi):\Cot M\rightarrow\Cot N$
as in Proposition \ref{pro: generating functions}. An enhancement
$[a]$ of the symplectic micromorphism $([V],\phi)$ yields a half
density germ\begin{eqnarray*}
[a_{\Psi}] & := & [(\iota_{\phi}^{f})^{*}a]\in|\Omega|^{\frac{1}{2}}\big(\phi^{*}(\Cot M)\big),\end{eqnarray*}
where $[\iota_{\phi}^{f}]:\phi^{*}(\Cot M)\rightarrow\Cot(M\times N)$
is the lagrangian embedding germ given by the generating family. Now,
we can associate to the triple $T=([a],[V],\phi)$ a semiclassical
Fourier integral operator\begin{eqnarray*}
Q_{\hbar}(T):\xH(M) & \longrightarrow & \xH(N)\end{eqnarray*}
exactly as we did for cotangent lifts, except that we replace
$S_{\phi}$ by its deformed version $S_{\phi}^{f}$ and $a$ by $a_{\Psi}$
in the semiclassical integral \eqref{eq:QuantFormula}. The crucial
point is that both $S_{\phi}$ and $S_{\phi}^{f}$ have a the same
unique critical point in the integration fiber $p_{N}^{-1}(x_{2})$,
namely $(0,\phi(x_{2}),x_{2})$. Moreover, the deformed family \[
\Big([B_{\phi}^{\Psi},Z_{\phi}],[p_{\phi}],[S_{\phi}^{f}]\Big)\]
continues to satisfy the assumptions of paragraph \ref{sub:Integral-representation}. 

\subsection{Local theory\label{sub:Local-theory}}

\subsubsection{Quantization of symplectic micromorphisms}

We are interested here in the enhancement and the quantization of
symplectic micromorphisms\begin{eqnarray*}
([L],\phi):\Cot U & \longrightarrow & \Cot V\end{eqnarray*}
between cotangent bundles of some open subsets $U\subset\R^{k}$ and
$V\subset\R^{l}$, whose canonical global coordinates we denote by
$(p_{1},x_{1})$ and $(p_{2},x_{2})$ respectively. As already noticed
in Examples \ref{exa:rhoB} and \ref{exa:Operator}, this case has
the special special feature of having a global exponential\begin{eqnarray*}
\Psi_{x}(v): & = & x+v\end{eqnarray*}
for each open subset $U\subset\R^{n}$. Hence, we can canonically
associate a generating family to any symplectic micromorphism $([L],V)$
from $\Cot U$ to $\Cot V$. Namely, Proposition \ref{pro: generating functions}
shows that there is a unique function germ $[f]\in C_{0}^{\infty}\Big(\phi^{*}(\Cot U),Z\Big)$
with $\partial_{p}f(0,x)=0$ such that\[
\Big([B(U,V),Z_{\phi}],[p_{\phi}],[S_{\phi}^{f}]\Big)\]
is a generating family for $([L],\phi)$, where \begin{eqnarray*}
B(U,V) & = & (\R^{k})^{*}\times U\times V,\\
Z_{\phi} & = & \{0\}\times\graph\phi,\\
S_{\phi}^{f}(p_{1},x_{1},x_{2}) & = & \big\langle p_{1},\phi(x_{2})-x_{1}\big\rangle+f(p_{1},x_{2}).\end{eqnarray*}

\begin{rem}
The generating function $S_{\phi}^{f}$ above differs by a minus sign
from our previous definition in \eqref{eq:deformed-generating-function}.
This sign is irrelevant, and we choose to write the generating function
this way here in order to better agree with the usual sign convention
for the phase of pseudo-differential operators. \end{rem}
\begin{notation}
At times, it will prove useful to collect the terms in $S_{\phi}^{f}$
that do not depend on $x_{1}$ into the single term\begin{eqnarray}
F(p_{1},x_{2}): & = & \langle p_{1},\phi(x_{2})\rangle+f(p_{1},x_{2}).\label{eq:Ff}\end{eqnarray}
The use of the upper case letter $F$ in the generating function notation
$S_{\phi}^{F}$ instead of the lower case $f$ will mean that we consider
the generating function\begin{eqnarray*}
S_{\phi}^{F}(p_{1},x_{1},x_{2}) & = & -\langle p_{1},x_{1}\rangle+F(p_{1},x_{2}),\end{eqnarray*}
in which the upper case $F$ is related to the lower case $f$ via
the relation \eqref{eq:Ff}. We write $S_{\phi}$ in case $F(p_{1},x_{2})=\langle p_{1},\phi(x_{2})\rangle$,
that is for the generating function of the cotangent lifts. Since,
the generating family is canonical in the local setting, we will also
use the notation $[S_{\phi}^{F}]$ to denote the corresponding symplectic
micromorphism $([L^{F}],\phi)$.
\end{notation}
The critical microfold of $S_{\phi}^{F}$ is\begin{eqnarray*}
\Sigma_{\phi}^{F} & = & \bigg\{\Big(p_{1},\partial_{p}F(p_{1},x_{2}),x_{2}\Big):\;(p_{1},x_{2})\in W\bigg\},\end{eqnarray*}
where $W$ is a suitable neighborhood of the zero of $\phi^{*}(\Cot U)$.
This yields the explicit formula \begin{eqnarray*}
i_{\phi}^{F}(p_{1},x_{2}) & = & \bigg(\Big(-\partial_{x_{1}}S_{\phi}^{F}(z),x_{1}\Big),\Big(\partial_{x_{2}}S_{\phi}^{F}(z),x_{2}\Big)\bigg),\\
 & = & \bigg(\Big(p_{1},\partial_{p}F(p_{1},x_{2})\Big),\Big(\partial_{x}F(p_{1},x_{2}),x_{2}\Big)\bigg),\end{eqnarray*}
where $z$ is the image of $(p_{1},x_{2})$ in $\Sigma_{\phi}^{F}$,
for the embedding of $[\phi^{*}(\Cot U),Z_{\phi}]$ into $\overline{\Cot U}\times\Cot V$. 

In the local case, an enhancement of $[S_{\phi}^{F}]$ can be identified
with a germ $[a]$ of semiclassical square integrable function \[
a\in L_{\hbar}^{2}((\R^{m})^{*}\times V)\]
around the $\{0\}\times V$, and the semiclassical intrinsic Hilbert
space $\xH(U)$ with $L_{\hbar}^{2}(U)$ in the notation of paragraph
\ref{sub:The-semiclassical-intrinsic}. The quantization of the enhanced
symplectic micromorphism $([a],S_{\phi}^{F})$ is the semiclassical
Fourier integral operator from $L_{\hbar}^{2}(U)$ to $L_{\hbar}^{2}(V)$
given by the formula

\begin{equation}
\Big(Q_{\hbar}\big([a],S_{\phi}^{f}\big)\Psi\Big)(x_{2})  :=  \SCint_{\R^{m}\times U}a(p_{1},x_{2})\Psi(x_{1})e^{\frac{i}{\hbar} S_\phi^f(p_1 , x_1 x_2)}\frac{dp_{1}dx_{1}}{(2\pi\hbar)^{m}},\label{eq:local-integral-representation}
\end{equation}
where $S_\phi^f(p_1, x_1, x_2) = \langle p_{1},\phi(x_{2})-x_{1}\rangle+f(p_{1},x_{2})$.
Note that these operators comprise the class
of semiclassical pseudo-differential operators. Namely, when both
$U$ and $V$ are $\R^{n}$ , the core map $\phi$ is the identity
and  $f=0$, we obtain the well-know integral representation
for semiclassical pseudo-differential operators: \begin{eqnarray*}
\big(Q_{\hbar}([a],S_{\id})\Psi\big)(x) & =\frac{1}{(2\pi\hbar)^{n}} & \SCint_{\R^{n}}a(p,y)f(y)e^{\frac{i}{\hbar}\langle p,x-y\rangle}dpdy,\end{eqnarray*}
where $a$ is the total symbol of the operator. Hence, this defines
a map\begin{eqnarray*}
\xOp_{\hbar}:[a] & \longmapsto & Q_{\hbar}([a],S_{\id}^{0})\end{eqnarray*}
from function germs on $\Cot\R^{n}$ around the the zero section to
operators acting on $L_{\hbar}^{2}(\R^{n})$. This is known as the
standard quantization of the symbol $a$, which satisfies the correspondence
principle of quantum mechanics:\begin{eqnarray*}
\xOp_{\hbar}([x])\Psi(x) & = & \frac{\hbar}{i}\frac{\partial\Psi}{\partial x}(x),\\
\xOp_{\hbar}([p])\Psi(x) & = & x\Psi(x),\end{eqnarray*}
where $[p]$ denotes the germ of the projection $(p,x)\mapsto p$ and
$[x]$ the germ of the projection $(p,x)\mapsto x$ around the zero
section. 

Here are other examples which are not pseudo-differential
operators.

The first example describes the class of semiclassical FIO's whose
wave fronts are the symplectic micromorphisms from $\Cot\R^{n}$ to
the point. 
\begin{example}
\label{Ex: counit}Consider a micromorphism $([L],\phi)$ from $\Cot\R^{n}$
to the cotangent bundle of the point, which we denote by $\star$. The core map $\phi:\{\star\}\rightarrow\R^{n}$ is completely determined
by the choice of a point $x_{0}\in\R^{n}$, and $[L]$ is a lagrangian
submanifold germ through $x_{0}$ transversal to the zero section.
Since $\phi^{*}(\Cot\R^{n})$ is the fiber $\Cot_{x_{0}}\R^{n}$,
the generating function is a function germ $[S_{x_{0}}^{f}]:[\Cot_{x_{0}}\R^{n},0]\rightarrow[\R,0]$
of the form \begin{eqnarray*}
S_{x_{0}}^{f}(p) & = & -\langle p,x_{0}\rangle+f(p).\end{eqnarray*}
A enhancement $[a]$ in this case is also a function germ on $\Cot_{x_{0}}U$
at the origin, and the corresponding quantization \begin{eqnarray*}
Q_{\hbar}([a],S_{x_{0}}^{f}):L_{\hbar}^{2}(\R^{n}) & \longrightarrow & \C[[\hbar]]\end{eqnarray*}
is a distribution on $\R^{n}$. When $f=0$, we obtain that\begin{eqnarray*}
Q_{\hbar}([a],S_{x_{0}})\Psi & = & \Big(\mathcal{F}_{\hbar}^{-1}\big(a\mathcal{F}_{h}(\Psi)\big)\Big)(x_{0}),\end{eqnarray*}
where $\mathcal{F}_{\hbar}$ is the asymptotic Fourier transform on
$\R^{n}$. In particular, this yields the derivatives of the delta
function concentrated at $x_{0}$: \begin{eqnarray*}
Q_{\hbar}(p^{\alpha},S_{x_{0}}) & = & (-i\hbar)^{|\alpha|}\partial^{\alpha}\delta(x-x_{0}),\end{eqnarray*}
where $\alpha\in\N^{n}$ is a multi-index. 
\end{example}

The second example describes the class of semiclassical FIO's whose
wave fronts are the unique symplectic micromorphism from the point
to $\Cot\R^{n}$. 
\begin{example}
\label{Ex: unit}There is only one symplectic micromorphism $e_{U}$
from  $\Cot \neutral$  (where $\neutral = \{\star\}$ is the point) to $\Cot U$. Explicitly, $e_{U}$ is given by $\big([\{0\}\times U],\pr\big)$,
where $\pr$ is the projection of $U$
to the unique point of
$E$. The generating function of $e_{U}$ is the zero function $S_{\pr}^{0}(x)=0$.
An enhancement of $e_{U}$ is, therefore, a function $f\in L_{\hbar}^{2}(U),$
and its quantization\begin{eqnarray*}
Q_{\hbar}(f,e_{U}):\C[[\hbar]] & \longrightarrow & L_{\hbar}^{2}(U)\end{eqnarray*}
can be identified with itself: $Q_{\hbar}(f,e_{U})\Psi=f\Psi$. 
\end{example}

\subsubsection{Generating function composition formula}

Consider  $U$, $V$, and $W$  open subsets of euclidean spaces, and let $(B(U,V),S_{\phi}^{F})$ and $(B(V,W),S_{\psi}^{G})$ be the generating
families generating the symplectic micromorphisms\[
\Cot U\overset{([L_{1}],\phi)}{\longrightarrow}\Cot V\overset{([L_{2}],\psi)}{\longrightarrow}\Cot W,\]
and let $\big(B(U,W),S_{\phi\circ\psi}^{G\circ F}\big)$ be the local
generating family of the composition $\big([L_{2}\circ L_{1}],\phi\circ\psi\big)$.
Obviously, symplectic micromorphism composition induces a composition operation on
local generating families\begin{eqnarray*}
S_{\psi}^{G}\circ S_{\phi}^{F} & := & S_{\phi\circ\psi}^{G\circ F}.\end{eqnarray*}
This composition is associative in the sense that it produces a category
in which objects are open subsets $U$ of $\R^{n}$ for some $n$ and
a morphism from $U$ to $V$ is a local generating family $\big(B(U,V),S_{\phi}^{F}\big)$.
The identity morphism on $U$ is the local generating family $\big(B(U,U),S_{\id}^{I}\big),$
where the generating function is the usual phase of pseudo-differential
operators\begin{eqnarray*}
S_{\id}^{I}(p_{1},x_{1},x_{2}) & = & I(p_{1},x_{2})-p_{1}x_{1},\end{eqnarray*}
where $I(p_{1},x_{2})=p_{1}x_{2}$. This category inherits a symmetric monoidal
 structure from the microsymplectic category. The unit object
is the point $\R^{0}=\{\star\}$, the tensor product on objects is
the usual cartesian product $U\times V$, and the tensor product of
local generating families\[
\Big(B(U_{1},V_{1}),S_{\phi_{1}}^{F_{1}}\Big)\otimes\Big(B(U_{2},V_{2}),S_{\phi_{2}}^{F_{2}}\Big)\]
is the local generating family \[
\Big(B\big(U_{1}\times U_{2},V_{1}\times V_{2}\big),S_{\phi_{1}}^{F_{1}}\otimes S_{\phi_{2}}^{F_{2}}\Big),\]
where the tensor product of the generating functions is given by
\begin{eqnarray*}
\big(S_{\phi}^{F}\otimes S_{\psi}^{G}\big)(p_{1},\tilde{p}_{2},x_{1},x_{2},\tilde{x}_{2},\tilde{x}_{3}) & := & S_{\phi}^{F}(p_{1},x_{1},x_{2})+S_{\psi}^{G}(\tilde{p}_{2},\tilde{x}_{2},\tilde{x}_{3}).
\end{eqnarray*}
The unit object $\R^{0}$ is initial: There is only one morphism from
$\R^{0}$ to any other open subset $U$. Namely, $B(\R^{0},U)=U$
and the only generating function in $G(\R^{0},U)$ is the zero function
that we denote by $S_{U}^{0}(x)=0.$ 

We now derive an explicit formula for the generating function $S_{\phi\circ\psi}^{G\circ F}$.
We will use two general facts on composition and reduction of generating
families:

\medskip{}

\textbf{Composition}. Let $p_{i}:B_{i}\rightarrow M_{i}\times M_{i+1}$
be a generating family for the canonical relation $L_{i}\subset\overline{\Cot M_{_{i}}}\times\Cot M_{i+1}$
with generating function $S_{i}$ for $i=1,2$. Suppose that the composition
of $L_{1}$ and $L_{2}$ is transversal. Then, the fibration $B_{1}\times_{M_{2}}B_{2}\rightarrow M_{1}\times M_{3}$
together with the generating function $(S_{1}+S_{2})(b_{1},b_{2})=S_{1}(b_{1})+S_{2}(b_{2})$
is a generating family for the canonical relation $L_{2}\circ L_{1}$. 

\medskip{}

\textbf{Reduction.} Let $p:B\rightarrow M$ be a generating family
with generating function $S$. Suppose that we can factor $p$ as a composition 
$B\overset{\pi}{\longrightarrow}B'\overset{p'}{\longrightarrow}M$ of  two
fibrations 
such that the restriction of $S$ to each fiber $\pi^{-1}(b')$ has the unique
critical point $\gamma(b')$. Then, the two generating families $(B,p,S)$
and $(B',p',S\circ\gamma)$ generate the same lagrangian submanifold
of $\Cot M$.

\medskip{}

With this in mind, we obtain that the fibration
\begin{eqnarray*}
B(U,V)\times_{V}B(V,W) & \overset{p_{V}}{\longrightarrow} & U\times W\end{eqnarray*}
with generating function\begin{eqnarray*}
\big(S_{\phi}^{F}\otimes_{V}S_{\psi}^{G}\big)(p_{1},p_{2},x_{1},x_{2},x_{2},x_{3}) & := & S_{\phi}^{F}(p_{1},x_{1},x_{2})+S_{\psi}^{G}(p_{2},x_{2},x_{3}),\end{eqnarray*}
is a generating family for $[L_{\psi}^{G}\circ L_{\phi}^{F}]$. Now,
we can factor the fibration $p_{V}$ into\[
B(U,V)\times_{V}B(V,W)\overset{\pi}{\longrightarrow}B(U,W)\overset{\pi_{U}\times\pi_{W}}
{\longrightarrow}U\times W.\]
Assuming that the restriction of $S_{\phi}^{F}\otimes_{V}S_{\psi}^{G}$
to each fiber $\pi^{-1}(p_{1},x_{1},x_{3})$ has a unique critical
point $\gamma(p_{1},x_{1},x_{3})$, we obtain a formula for the generating
function composition\begin{eqnarray*}
S_{\psi}^{G}\circ S_{\phi}^{F} & := & (S_{\phi}^{F}\otimes_{V}S_{\psi}^{G})\circ\gamma.\end{eqnarray*}
The following lemma guarantees the existence of a unique critical
point on each fiber. 
\begin{lem}
In the notation as above, we have that the restriction of $S_{\phi}^{F}\otimes_{V}S_{\psi}^{G}$
to the fiber $\pi^{-1}(p_{1},x_{1},x_{3})$ has a unique non-degenerate
critical point\begin{eqnarray*}
(\bar{p}_{2},\bar{x}_{2}) & = & \gamma(p_{1},x_{1},x_{3})\end{eqnarray*}
which is given as the unique solution of system\begin{eqnarray}
\overline{p}_{2} & = & \partial_{x}F(p_{1},\overline{x}_{2}),\label{eq:CritPointP}\\
\bar{x}_{2} & = & \partial_{p}G(\overline{p}_{2},x_{3}).\label{eq:CritPointX}\end{eqnarray}
\end{lem}
\begin{proof}
By definition, we have:\begin{eqnarray*}
(S_{\phi}^{F}\otimes_{V}S_{\psi}^{G})(p_{1},x_{1},x_{2},p_{2},x_{3}) & = & p_{1}x_{1}+F(p_{1},x_{2})+G(p_{2},x_{3})-p_{2}x_{2}.\end{eqnarray*}
Now, the critical points along $\pi^{-1}(p_{1},x_{1},x_{3})=(\R_{V}^{l})^{*}\times V$
are the points $(\bar{p}_{2},\bar{x}_{2})$ in the fiber satisfying
the equation\begin{eqnarray*}
\frac{\partial\big(S_{\phi}^{F}\otimes_{V}S_{\psi}^{G}\big)}{\partial(p_{2},x_{2})}\big(p_{1},x_{1},\bar{x}_{2},\bar{p}_{2},x_{3}\big) & = & 0,\end{eqnarray*}
which is exactly equivalent to the system \eqref{eq:CritPointP}-\eqref{eq:CritPointX},
since\[
\frac{\partial\big(S_{\phi}^{F}\otimes_{V}S_{\psi}^{G}\big)}{\partial(p_{2},x_{2})}\big(p_{1},x_{1},\bar{x}_{2},\bar{p}_{2},x_{3}\big)=\left(\begin{array}{c}
\bar{p}_{2}-\partial_{x_{2}}F(p_{1},\bar{x}_{2})\\
\bar{x}_{2}-\partial_{p_{2}}G(\bar{p}_{2},x_{3})\end{array}\right).\]
Thanks to the fact that $F(0,x_{1})=0$ and $\partial_{p_{2}}G(0,x_{3})=\psi(x_{3})$,
we get\[
\frac{\partial\big(S_{\phi}^{F}\otimes_{V}S_{\psi}^{G}\big)}{\partial(p_{2},x_{2})}\big(0,x_{1},\psi(x_{3}),0,x_{3}\big)=0\]
for all $x_{1}$, meaning that $\gamma(0,x_{1},x_{3})=(0,\psi(x_{3}))$.
In particular, this is true for $x_{1}=\phi\circ\psi(x_{3})$. Now, the Hessian
at this point is \[
\frac{\partial^{2}\big(S_{\phi}^{F}\otimes_{V}S_{\psi}^{G}\big)}{\partial^{2}(p_{2},x_{2})}\big(0,\phi\circ\psi(x_{3}),\psi(x_{3}),0,x_{3}\big)=\left(\begin{array}{cc}
\id & 0\\
-\partial_{p}\partial_{x}G(0,x_{3}) & \id\end{array}\right),\]
which is invertible. Now, for small $p_{1},$ the implicit function
theorem gives us a section of critical points \begin{eqnarray*}
\big[B(V,W),Z_{\phi\circ\psi}\big] & \overset{\gamma}{\longrightarrow} & \Big[(B(U,V)\times_{V}B(V,W),Z_{\phi}\times_{V}Z_{\psi}\Big].\end{eqnarray*}
\end{proof}
\begin{defn}
Given a function $f\in C^{\infty}(\R^{k})$ with only one critical
point on $\R^{k}$, we denote by $\stat{x}\{f\}$ the value of $f$
at this point $x_{0}$ where $\partial_{x}f(x_{0})=0.$ If $f$ also
depends on a variable $y$ in $\R^{l}$, we denote by $\stat{x}\{f\}(y)$
the function depending on $y$ defined by $f(x_{0}(y),y)$ where $x_{0}(y)$
is the implicit function solution of the equation $\partial_{x}f(x_{0}(y),y)=0$. 
\end{defn}
We now immediately have the following proposition:
\begin{prop}
 \begin{eqnarray*}
(G\circ F)(p_{1},x_{3}) & = & \stat{\bar{p},\bar{x}}\Big\{ F(p_{1},\bar{x})+G(\bar{p},x_{3})-\bar{p}\bar{x}\Big\}.\end{eqnarray*}
Moreover,\[
F\,=\, I\circ F\,=F\circ I,\]
where $I(p,x)=\langle p,x\rangle$. 
\end{prop}

\subsubsection{Enhancement composition formula\label{sub:Enhancement-composition-formula}}

We denote by $E(K)$ the space of enhancements of a symplectic micromorphism
$K:\Cot U\rightarrow\Cot V$.  In the local setting, the micromorphism
$K$ is associated to a unique generating family $(B(U,V),S_{\phi}^{F})$.
Suppose now that we have another symplectic micromorphism $L:\Cot V\rightarrow\Cot W$
with generating family $(B(V,W),S_{\psi}^{G})$. We want to define
a composition for enhancements\begin{eqnarray*}
\circ:E(K)\times E(L) & \longrightarrow & E(L\circ K).\end{eqnarray*}
In the previous paragraph, we have seen that we have a fibration\begin{eqnarray*}
B(U,V)\times_{V}B(V,W) & \overset{\pi}{\longrightarrow} & B(U,W)\end{eqnarray*}
on whose fiber the generating function $S_{\phi}^{F}\otimes_{V}S_{\psi}^{G}$ has a single non degenerate critical point.
The evaluation of $S_{\phi}^{F}\otimes S_{\psi}^{G}$ at this critical point is exactly $S_{\psi}^{G}\circ S_{\phi}^{F}$.
For $[a]\in E(K)$ and $[b]\in E(L)$, we define the enhancement composition
as
\begin{equation}
\big(a\circ b\big)(p_{1},x_{3}) := \SCint_{\pi^{-1}(p_{1},\phi\circ\psi(x_{3}),x_{3})}a(p_{1},x_{2})b(p_{2},x_{3})e^{\frac{i}{\hbar}\Theta(F,G)}\frac{dp_{2}dx_{2}}{(2\pi\hbar)^{l}},\label{eq:half-density composition}
\end{equation}
where $\Theta(F,G)$ is defined as the difference\begin{eqnarray*}
\Theta(F,G) & = & S_{\phi}^{F}\otimes_{V}S_{\psi}^{G}-\pi^{*}(S_{\psi}^{G}\circ S_{\phi}^{F}).\end{eqnarray*}

\begin{rem}
Observe that the leading term in $\hbar$ of the stationary phase
asymptotic expansion of $a\circ b$ is proportional to $a(p_{1},\overline{x}_{2})b(\overline{p}_{2},x_{3})$
where $(\bar{p}_{2},\bar{x}_{2})$ is the critical point of $S_{\phi}^{F}\otimes S_{\psi}^{G}$
on the fiber $\pi^{-1}(p_{1},\psi\circ\phi(x_{3}),x_{3})$. This is
the usual composition of half-densities carried by composable canonical
relations. \end{rem}
\begin{prop}
In the notation above, we have that \begin{eqnarray*}
([a],S_{\phi}^{F}) & = & ([1],S_{\id}^{I})\circ([a],S_{\phi}^{F})\\
 & = & ([a],S_{\phi}^{F})\circ([1],S_{\id}^{I})\end{eqnarray*}
for any enhanced symplectic micromorphism $([a],S_{\phi}^{F})$. 
\end{prop}

\begin{proof}
We will show only the first equality, since the second is proven similarly.
Recall that $I(p,x)=px$ and that
$I\circ F=F=F\circ I.$ We have that \begin{eqnarray*}
\big(\Theta(I,F)\big)(z) & = & p_{1}\phi(x_{3})+F(p_{1},x_{2})+I(p_{2},x_{3})-p_{2}x_{2}-p_{1}\phi(x_{3})-\big(F\circ I\big)(p_{1},x_{3})\\
 & = & F(p_{1},x_{2})+p_{2}(x_{3}-x_{2})-F(p_{1},x_{3}),\end{eqnarray*}
where $z=(p_{1},\phi(x_{3}),x_{2},p_{2},x_{2},x_{3}).$ Plugging this
into the enhancement composition formula  (\ref{eq:half-density composition}), we obtain
\begin{eqnarray*}
\big(1\circ a\big)(p_{1},x_{3}) & = & \frac{e^{-\frac{i}{\hbar}F(p_{1},x_{3})}}{(2\pi\hbar)^{l}}\int dx_{2}\,\, a(p_{1},x_{2})e^{\frac{i}{\hbar}F(p_{1},x_{2})}\int dp_{2}\,\, e^{\frac{i}{\hbar}p_{2}(x_{3}-x_{2})}.\\
 & = & e^{-\frac{i}{\hbar}F(p_{1},x_{3})}\int dx_{2}\,\, a(p_{1},x_{2})e^{\frac{i}{\hbar}F(p_{1},x_{2})}\delta_{x_{3}}(x_{2})\\
 & = & a(p_{1},x_{3}),
\end{eqnarray*}
where $\delta_{x_{3}}$ is the delta function concentrated at $x_{3}$.
\end{proof}

\begin{prop}
Let $E_{1}$ and $E_{2}$ be any two composable enhanced symplectic micromorphisms.
Then: \begin{eqnarray*}
Q_{\hbar}(E_{2})\circ Q_{\hbar}(E_{1}) & = & Q_{\hbar}(E_{2}\circ E_{1})\quad\mod\mathcal{O}(\hbar^{\infty}).\end{eqnarray*}
 Moreover, if
$Q_{\hbar}([a],L)=0$, then $[a]=0\mod\mathcal{O}(\hbar^{\infty})$\end{prop}
\begin{proof}
Suppose we have  enhanced symplectic micromorphisms $E_{1}=([a],S_{\phi}^{F})$
and $E_{2}=([b],S_{\psi}^{G})$ in the situation\[
\Cot U\overset{E_{1}}{\longrightarrow}\Cot V\overset{E_{2}}{\longrightarrow}\Cot W.\]
A straightforward computation yields\begin{eqnarray*}
I & = & \Big(Q_{\hbar}(E_{2})\circ Q_{\hbar}(E_{1})\Psi\Big)(x_{3})\\
 & = & \int_{\R^{l}\times V}\frac{dp_{2}dx_{2}}{(2\pi\hbar)^{l}}b(p_{2},x_{3})\big(Q_{\hbar}(E_{1})\Psi\big)(x_{2})e^{\frac{i}{\hbar}S_{\Psi}^{G}(p_{2},x_{2},x_{3})},\\
 & = & \int_{\R^{l}\times V}\frac{dp_{2}dx_{2}}{(2\pi\hbar)^{l}}\int_{\R^{k}\times U}\frac{dp_{1}dx_{1}}{(2\pi\hbar)^{k}}a(p_{1},x_{2})b(p_{2},x_{3})\Psi(x_{1})e^{\frac{i}{\hbar}\big(S_{\Psi}^{G}(p_{2},x_{2},x_{3})+S_{\phi}^{F}(p_{1},x_{2},x_{3})\big)}.\end{eqnarray*}
Now, interchanging the integrals and writing 1 as\begin{eqnarray*}
1 & = & e^{\frac{i}{\hbar}\big(S_{\psi}^{G}\circ S_{\phi}^{F}\big)(p_{1},\phi\circ\psi(x_{3}),x_{3})}e^{-\frac{i}{\hbar}\big(S_{\psi}^{G}\circ S_{\phi}^{F}\big)(p_{1},\phi\circ\psi(x_{3}),x_{3})},\end{eqnarray*}
we obtain that\begin{eqnarray*}
I & = & \int_{\R^{k}\times U}\frac{dp_{1}dx_{1}}{(2\pi\hbar)^{k}}(a\circ b)(p_{1},x_{3})\Psi(x_{1})e^{\frac{i}{\hbar}\big(S_{\psi}^{G}\circ S_{\phi}^{F}\big)(p_{1},\phi\circ\psi(x_{3}),x_{3})},\\
 & = & Q_{\hbar}(E_{2}\circ E_{1}).\end{eqnarray*}
Of course, all the computations should be understood modulo $\mathcal{O}(\hbar^{\infty})$
and with the appropriate cut-off functions thrown in. The fact that
$Q_{\hbar}([a],L)=0$ implies that $[a]=0$ mod $\hbar^{\infty}$
is obvious. \end{proof}

\begin{cor}
Enhanced local symplectic micromorphisms form a category.\end{cor}
\begin{proof}
Let $([a_{i}],T_{i})$ , $i=1,2,3$, be composable enhanced local
symplectic micromorphisms, and set $Q_{i}:=Q_{\hbar}([a_{i}],T_{i})$
be their quantization. We need to prove that\begin{eqnarray*}
[a_{3}\circ(a_{2}\circ a_{1})] & = & [a_{3}\circ(a_{2}\circ a_{1})].\end{eqnarray*}
We know that associativity holds at the level of quantization:\begin{eqnarray*}
Q_{3}(Q_{2}Q_{1}) & = & (Q_{3}Q_{2})Q_{1}.\end{eqnarray*}
Now, we also have that \begin{eqnarray*}
Q_{3}(Q_{2}Q_{1}) & = & Q_{\hbar}\Big([(a_{3}\circ(a_{2}\circ a_{1})],T_{3}\circ T_{2}\circ T_{1}\Big),\\
(Q_{3}Q_{2})Q_{1} & = & Q_{\hbar}\Big([(a_{3}\circ a_{2})\circ a_{1})],T_{3}\circ T_{2}\circ T_{1}\Big),\end{eqnarray*}
and, therefore, $[a_{3}\circ(a_{2}\circ a_{1})]=[a_{3}\circ(a_{2}\circ a_{1})]\mod\hbar^{\infty}.$ \end{proof}
\begin{example}
\label{exa:Moyal}We have see that the quantization $Q_{\hbar}([a],S_{\id})$
of the enhanced symplectic micromorphism\begin{eqnarray*}
([a],\Cot\id):\Cot\R^{n} & \longrightarrow & \Cot\R^{n}\end{eqnarray*}
coincides with the standard quantization $\xOp_{\hbar}(a)$ of $a$
seen as a symbol (germ) in $C^{\infty}(\Cot\R^{n})$. Therefore,
the enhancement composition degenerates in this case to the star
product $\star_{qp}$ defined by the identity
\begin{eqnarray*}
\xOp_{\hbar}(a)\circ\xOp_{\hbar}(b) & = & \xOp_{\hbar}(a\star_{qp}b),
\end{eqnarray*}
giving the composition of symbols in the qp-ordering for pseudodifferential operators.
\end{example}

The next example shows that one can recover the quantization out of the
enhancement composition; so, in our case, enhancing is quantizing.

\begin{example}
Enhancements $(\Psi,\e_{U})$ of the unique symplectic micromorphism
$\e_{U}:\neutral\rightarrow\Cot U$ can be identified with the space
$\Psi\in L_{\hbar}^{2}(U)$. Let $([a],T)$ be a symplectic micromorphism
from $\Cot U$ to $\Cot U$. An easy computation shows that composition
degenerates into quantization; namely, we have that \begin{eqnarray*}
[a]\circ\Psi & = & Q_{\hbar}([a],T)\Psi.\end{eqnarray*}

\end{example}

\subsubsection{States and costates}
We call a \textbf{state} on $U$ an enhanced symplectic micromorphism
$(\Psi,e_{U})$ from the cotangent microbundle of the point $\mathbf E$ to
$\Cot U$. As is clear from Example \ref{Ex: unit}, the set of states
on $U$ can be identified with the Hilbert space $Q_{\hbar}(\Cot U)=L_{\hbar}^{2}(U)$,
and we will use the Dirac {}``ket'' notation\begin{eqnarray*}
(\Psi,e_{U}) & \leftrightsquigarrow & |\Psi\rangle,\end{eqnarray*}
for states. An enhanced symplectic micromorphism $A=([a],S_{\phi}^{F})$
from $\Cot U$ to $\Cot V$ induces an operator $\hat{A}$ from the
states on $U$ to the states on $V$ by  composition\begin{eqnarray*}
\hat{A}|\Psi\rangle & := & A\circ(\Psi,e_{U}).\end{eqnarray*}
The following proposition tells us that the quantization functor stems
from the composition of enhancements:
\begin{prop}
Let $A:\Cot U\rightarrow\Cot V$ be an enhanced symplectic micromorphism,
the\textup{n\begin{eqnarray*}
\hat{A}\,|\Psi\rangle & =|Q_{\hbar}(A)\Psi\rangle & ,\end{eqnarray*}
}where $\Psi$\textup{ is a state on $U$.}\end{prop}
\begin{proof}
Set $A=([a],S_{\phi}^{F})$. Then, we have by definition that $\hat{A}|\Psi\rangle=(a\circ\Psi,e_{V})$.
Now, a straightforward computation yields that $\Theta(0,F)=S_{\phi}^{F}$,
which in turns gives us that\begin{eqnarray*}
a\circ\Psi(x_{2}) & = & \SCint\Psi(x_{1})a(p_{1},x_{2})e^{\frac{i}{\hbar}S_{\phi}^{F}(p_{1},x_{1},x_{2})}\frac{dp_{1}dx_{1}}{(2\pi\hbar)^{l}}\\
 & = & Q_{\hbar}(A)\Psi(x_{2}).\end{eqnarray*}

\end{proof}
A \textbf{costate} on $U$ is an enhanced symplectic micromorphism
from $\Cot U$ to $E$. It is completely determined by the data
of a point $x\in U$, a germ of a lagrangian submanifold $[L]$ around
$x$ in $\Cot U$ transversal to the zero section and an enhancement
$[a]:\Cot_{x}U\rightarrow\R$ as in Example \ref{Ex: counit}. We
will use the Dirac {}``bra'' notation\begin{eqnarray*}
\big([a],S_{x}^{F}\big) & \leftrightsquigarrow & \langle a,F,x|\end{eqnarray*}
to denote costates, where $S_{x}^{F}$ is the generating function
of $[V]$. Note that the set of costates does not itself form a vector
space; however, one may consider the free vector space generated by
costates. Using the composition of enhanced symplectic micromorphisms,
we can define a {}``pairing'' between states and costates on $U$:\begin{eqnarray*}
\big\langle a,F,x\big|\Psi\big\rangle & := & ([a],S_{x_{0}}^{F})\circ(\Psi,S_{\pr}^{0})\end{eqnarray*}
is a symplectic micromorphism from $E$ to $E$ and, thus, a power series
belonging to $\C[[\hbar]]$. 
\begin{example}
If we denote by $\langle x|$ the special costate $(1,0,x)$ on $\R$,
we obtain that \begin{eqnarray*}
\langle x_{0}|\Psi\rangle & = & \Psi(x_{0}),\end{eqnarray*}
and if we denote by $\hat{H}$ the operator on states associated to
the enhancement $([H],S_{\id}^{0}):\R\rightarrow\R$ of the identity
micromorphism, we obtain that\begin{eqnarray*}
\langle x_{0}|\hat{H}|\Psi\rangle & = & \mathcal{F}_{\hbar}^{-1}\big(H\mathcal{F}_{\hbar}(\Psi)\big)(x_{0}),\end{eqnarray*}
 and in particular \begin{eqnarray*}
\langle x_{0}|\hat{p}|\Psi\rangle & = & -i\hbar\frac{\partial\Psi}{\partial x}(x_{0}),\\
\langle x_{0}|\hat{x}|\Psi\rangle & = & x_{0}\Psi(x_{0}).\end{eqnarray*}
In view of Example \ref{exa:Moyal}, we have that $\hat{f}\circ\hat{g}=\widehat{f\star_{M}g}.$\begin{eqnarray*}
\end{eqnarray*}

\end{example}

\section{Applications and further directions\label{sec:Applications-and-further}}

In this section, we describe in an informal way some of the applications of symplectic micromorphism
enhancements and quantization. Roughly, symplectic microgeometry and
its quantized version offer a framework in which dynamics can be
expressed in a purely categorical way both at the classical and quantum
level. There is a special monoid in the microsymplectic category,
the energy monoids whose action on cotangent microbundles represent
the Hamiltonian flows of classical mechanics. Symmetries of space
are implemented by the action of general monoids in the microsymplectic
category and they correspond to the presence of a Poisson structure on
the core of the monoid. The enhancement and quantization of this picture
recovers the semi-classical version of the Schr\"odinger evolution.
Let us start with some consideration on the functoriality of our constructions, very closely related  the work of Khudaverdian and Voronov cited in the Introduction.

\subsection{Quantization functors}

In paragraph \ref{sub:Local-theory}, we restricted the category of
manifolds we started with to the category $\mathcal{E}$ formed by
the open subsets of the Euclidean spaces $\R^{n}$ endowed with the
canonical Euclidean metric. This extra data of a metric made it possible
to produce a canonical (global even!) exponential $\Psi_{x}(v)=v+x$
for each object in $\mathcal{E}$ since $\Psi$ is exactly the exponential
associated to the affine connection generated by the metric.
(Note that the only piece of extra data we need for our construction on the top of the manifold is the exponential $\Psi$ ; we do not make direct use of the metric or the affine structure.) This
allowed us to identify, in a non ambiguous way, the class of semiclassical
Fourier integral operators $\SFour_{\hbar}(\mathcal{E})$ with wave
fronts in the category $\Mic(\mathcal{E})$ formed by the symplectic
micromorphisms between the cotangent bundles on the object of $\mathcal{E}$
with the space of symplectic micromorphism enhancements. In symbols,
we had the identification\[
\SFour_{\hbar}([V],\phi)=|\Omega_{\hbar}|^{\frac{1}{2}}([V],\phi),\]
thanks to the canonical way we could associate to $([V],\phi)$ a
generating function $S_{\phi}^{f}$ and to a semiclassical FIO with
wave front $([V],\phi)$ the integral representation \eqref{eq:local-integral-representation}. 
This allowed us further to endow the corresponding collection $E(\mathcal{E})$
of enhanced symplectic micromorphisms with the structure of a category
by pulling back the FIO composition to their total symbols as in paragraph
\ref{sub:Enhancement-composition-formula}. This situation can be
summarized by the following commutative diagram of functors

\begin{diagram} 
         \SFour_\hbar(\mathcal E) && \pile{\rTo^{\sigma} \\ \lTo_{Q_\hbar}} &&  E(\mathcal E) \\
         & \rdTo_{\WF}            &                                          & \ldTo_{U} &\\
                                  &&          \Mic(\mathcal E)              &&
\end{diagram}
where $\sigma$ is the total symbol functor, $\WF$ the wave front
functor, $Q_{\hbar}$ the quantization functor, and $U$ the forgetful
functor. Observe that $\sigma$ is an isomorphism of categories with
inverse functor given by the quantization functor $Q_{\hbar}$ . 

We conclude this discussion by observing that the functorial constructions
above apply in the very exact same way to any other category $\mathcal{E}$
of manifolds carrying some extra data that make it possible to assign in a canonical
way a micro exponential to each of the manifolds. For instance, we
may take $\mathcal{E}$ to be the category of riemannian manifolds
and the canonical micro exponentials given by the germ of the Levi-Civita
connection exponential. Note also that the functor $Q_{\hbar}$ is
strictly monoidal with respect to the obvious monoidal structures
on the source and target categories. 

In the sequel, we will comment on possible applications of our construction,
sometimes assuming implicitly a suitable category $\mathcal{E}$ of
manifolds endowed with the appropriate extra data.

\subsection{The energy monoid}

The Lie algebra $\mathcal{T}$of the time translation group $(\R,+)$
is the abelian Lie algebra on $\R$. Its dual, which we denote by
$\mathcal{E}$, is thus the Poisson manifold, with zero Poisson structure. 

We call $\mathcal{T}$ the Lie algebra of time and $\mathcal{E}$
the Poisson manifold of energy. Its cotangent bundle $\Cot\Energy=\mathcal{T}\times\mathcal{E}$
is a trivial symplectic microgroupoid with source and target coinciding
with the bundle projection; the composable pair space is $\Cot\Energy\oplus\Cot\Energy$
and the groupoid product is just the addition of times in a fiber
of constant energy:\begin{eqnarray*}
m_{\Energy}\big((t_{1},E),(t_{2},E)\big) & = & (t_{1}+t_{1},E).\end{eqnarray*}
One verifies that the graph of the groupoid product is a symplectic
micromorphism\begin{eqnarray*}
\mu_{\Energy}:=\big(\graph[m_{\Energy}],\Delta_{\mathcal{E}}):\Cot\Energy\otimes\Cot\Energy & \longrightarrow & \Cot\Energy,\end{eqnarray*}
where the core map $\Delta_{\mathcal{E}}$ is the diagonal map on
$\mathcal{E}$. It is easy to see that $\mu_{\mathcal{E}}$ satisfies
the following associativity and unitality equations\begin{gather*}
\mu_{\Energy}\circ(\mu_{\Energy}\otimes\id)\;=\;\mu_{\Energy}\circ(\id\otimes\mu_{\Energy}),\\
\mu_{\Energy}\circ(e_{\Energy}\otimes\id)\;=\;\id\;=\mu_{\Energy}\circ(\id\otimes e_{\Energy}),\end{gather*}
where $ $$e_{\Energy}$ is the unique symplectic micromorphism from
the cotangent bundle of the point to $\Cot\Energy$. In other words,
$(\Cot\Energy,\mu_{\Energy})$ is a monoid object in the microsymplectic
category (see \cite{SM_Monoids} for a systematic study of these monoids).

\subsection{Classical flows as symplectic micromorphisms\label{sub:Classical-flow}}

Consider a classical hamiltonian system $H:\Cot Q\rightarrow\R$.
The time evolution $\Psi_{t}$ generated by $H$ produces a lagrangian
submanifold \[
W_{H}:=\bigg\{\Big(\big(t,H(\Psi_{t}(z))\big),z,\Psi_{t}(z)\Big):\; t\in I,\, z\in\Cot Q\bigg\},\]
of $\overline{\Cot\Energy}\times\overline{\Cot Q}\times\Cot Q$, where
$I$ is the maximal interval on which the flow $\Psi_{t}$ is defined.
It gives a symplectic micromorphism\begin{eqnarray*}
\rho_{H}:=\big([W_{H}],H_{|Z_{Q}}\times\id_{Q}\big):\Cot\Energy\otimes\Cot Q & \longrightarrow & \Cot Q.\end{eqnarray*}

\begin{prop}
Let $H\in C^{\infty}(\Cot Q)$ be a hamiltonian. Then
that \begin{eqnarray*}
\rho_{H}\circ(e_{\Energy}\otimes\id) & = & \id,\\
\rho_{H}\circ(\mu_{\Energy}\otimes\id) & = & \rho_{H}\circ(\id\otimes\rho_{H}).\end{eqnarray*}
In other words, this turns $\Cot Q$ into a $\Cot\Energy$-module
in the microsymplectic category.\end{prop}
\begin{proof}
A straightforward composition of canonical (micro) relations yields\begin{eqnarray*}
\rho_{H}\circ(\id\otimes\rho_{H}) & = & \bigg\{\bigg(t_{2},H\big(\Psi_{t_{1}}\circ\Psi_{t_{2}}(z)\big),t_{1},H\big(\Psi_{t_{1}}(z)\big),z,\Psi_{t_{2}}\circ\Psi_{t_{1}}(z)\bigg):\, z,t\bigg\},\\
\rho_{H}\circ(\mu_{\Energy}\otimes\id) & = & \bigg\{\bigg(t_{2},H\big(\Psi_{t_{1}+t_{2}}(z)\big),t_{1},H\big(\Psi_{t_{1}+t_{2}}(z)\big),z,\Psi_{t_{1}+t_{2}}(z)\bigg):\, z,t\bigg\}.\end{eqnarray*}
We see that these microfolds coincide for time independent hamiltonians,
since $\Psi_{t_{1}}\circ\Psi_{t_{2}}=\Psi_{t_{1}+t_{2}}$ and $H(\Psi_{t}(z))=H(z)$
for all $t_{1},t_{2}$ and $t$. One gets the unitality axiom in a
similar way.
\end{proof}
The Hamilton-Jacobi formulation of classical mechanics tells us that,
in  Darboux coordinates, the hamiltonian flow\begin{eqnarray*}
(p_{t},x_{t}) & = & \Psi_{t}^{H}(p,x)\end{eqnarray*}
 admits, for short times, a generating function $S$ satisfying \begin{eqnarray*}
\partial_{t}S(t,p,x_{t}) & = & H_{t}\big(p_{t},x_{t}\big),\\
\partial_{p}S(t,p,x_{t}) & = & x,\\
\partial_{x}S(t,p,x_{t}) & = & p_{t},\end{eqnarray*}
with the initial condition $S(0,p,x)=\langle p,x\rangle$. This is
precisely the generating function of the symplectic micromorphism
$\rho_{H}$.

\subsection{Classical symmetries}

It is possible to generalize the previous scheme to a general Hamiltonian
action of a Lie group $G$ on $\Cot P$ with momentum map $J:\Cot P\rightarrow\mathcal{G}^{*}$, where $\mathcal{G}$ is the Lie algebra of $G$ .
In this case, we define the symmetry submanifold to be\begin{eqnarray*}
W_{G} & := & \bigg\{\Big(\big(v,J(\exp(v)z)\big),z,\exp(v)z)\Big):\; v\in U,\, z\in\Cot Q\bigg\},\end{eqnarray*}
where $U$ is the maximal neighborhood of $0$ in the Lie algebra
$\mathcal{G}$ on which the exponential mapping $\exp:\mathcal{G}\rightarrow G$
is defined. Taking the germ of $W_{G}$ around the graph of $J_{|Q}\times\id_{Q}$,
yields a symplectic micromorphism $\rho_{G}$ from $ $$\Cot\mathcal{G}^{*}\otimes\Cot Q$ (where $\otimes$ denotes the tensor product on objects defined by the Cartesian product)
to $\Cot Q$. Now, thanks to the exponential mapping, we can define
a generating function germ from $\Cot\mathcal{G}^{*}\oplus\Cot\mathcal{G}^{*}$
to $\R$ via the formula\begin{eqnarray*}
S_{G}(v,w,\mu) & := & \Big\langle\mu,\exp^{-1}\big(\exp(v)\exp(w)\big)\Big\rangle,\end{eqnarray*}
where $\langle\,,\,\rangle$ is the canonical paring between the Lie
algebra and its dual. This generating function germ defines a symplectic
micromorphism $\mu_{G}$ from $\Cot\mathcal{G}^{*}\otimes\Cot\mathcal{G}^{*}$
to $\Cot\mathcal{G}^{*}$. One can show that $(\Cot\mathcal{G}^{*},\mu_{G})$
is a monoid and that $(\Cot Q,\rho_{G})$ is a $\Cot\mathcal{G}^{*}$-module.

\subsection{Generalized symmetries}

We can consider more general symmetries in microgeometry by allowing
arbitrary monoids $(\Cot P,\mu)$ to act on phase spaces $\Cot Q$.
It turns out that a general monoid induces a Poisson structure on
its core $P$, and, conversely, any Poisson manifold $(P,\Pi)$ induces
a monoid $(\Cot P,\mu)$ by considering the local symplectic groupoid
integrating $P$ and by taking $\mu$ to be the germ of the groupoid
product around the graph of the diagonal map on $P$ (\cite{SM_Monoids}).
Moreover, one can show that a $\Cot P$-module,\begin{eqnarray*}
\rho:\Cot P\otimes\Cot Q & \longrightarrow & \Cot Q,\end{eqnarray*}
induces a momentum map germ, i.e., a Poisson map germ $[J]$ from
the cotangent microbundle $\Cot Q$ to the Poisson manifold $P$.
Such situations may arise for instance if we start with a Lie groupoid
acting in an hamiltonian way on a phase space.

\subsection{Quantization of the energy monoid\label{sub:Quant of the energy monoid}}

A straightforward, but lengthy, computation would show that the following
enhancement $([1],\mu_{\Energy})$ of the energy monoid $(\Cot\Energy,\mu_{\Energy})$
yields a monoid in the category of enhanced symplectic microfolds.
More interesting is the quantization of this enhanced monoid: it should
produce a associative product on $\xH_{\Energy}$, that is on $C^{\infty}(\Energy)[[\hbar]]$.
Let us compute it:\begin{eqnarray*}
Q_{\hbar}(([1],\mu_{\Energy})\big(f\otimes g\big)(E) & = & \SCint\hat{f}(t_{1})\hat{g}(t_{2})e^{-\frac{i}{\hbar}S(t_{1},t_{2},E)}\frac{dt_{1}dt_{2}}{(2\pi\hbar)},\end{eqnarray*}
where $\hat{f}$ and $\hat{g}$ are the asymptotic Fourier transforms
of $f$ and $g$, and where $S$ is the generating function of the
energy monoid. Since $S=(t_{1}+t_{2})E$, we see that the quantization
of $([1],\mu_{\Energy})$ yields the extension of the usual product
of functions on $C^{\infty}(\Energy)[[\hbar]]$.

\subsection{Quantization of Poisson manifolds}

Suppose we have a monoid $(\Cot P,\mu)$ with induced Poisson structure
$\Pi$ on $P$. Since the core map of $\mu$ is the diagonal map $\Delta$
on $P$, an enhancement of $\mu$ is given by any half-density germ\[
[a]\in|\Omega_{h}|^{\frac{1}{2}}(\Cot P\oplus\Cot P)\]
around the zero-section. Given any such enhancement of $\mu$, we
obtain thus an operator \begin{eqnarray*}
\star_{\hbar}:=Q_{\hbar}([a],\mu):\xH(P)\otimes\xH(P) & \longrightarrow & \xH(P).\end{eqnarray*}
The functoriality of $Q_{\hbar}$ implies that $\star_{\hbar}$ is
an associative product if and only if the enhancement is associative,
in which case the associative algebra $(\xH(P),\star_{\hbar})$ should
be considered as the quantization of the Poisson manifold $(P,\Pi)$.
In this context, the quantization of Poisson manifolds becomes equivalent
to the existence of associative enhancements. When $P$ is an open
subset of $\R^{n}$ we have the identification of $\xH(P)$ with $C^{\infty}(P)[[\hbar]]$,
and we can write the product as \begin{eqnarray*}
f\star_{\hbar}g(x) & = & \SCint\hat{f}(p_{1})\hat{g}(p_{2})a(p_{1},p_{2},x)e^{\frac{i}{\hbar}S(p_{1},p_{2},x)}\frac{dp_{1}dp_{2}}{(2\pi)^{m}},\end{eqnarray*}
where $S$ is the generating function of $\mu$. Note that when the
Poisson structure is the zero Poisson structure, or, equivalently,
when $\mu$ is the graph of vector bundle addition in $\Cot P$, we
have that the generating function is $S=\langle p_{1}+p_{2},x\rangle$
as in the energy monoid case. The associative enhancement $([1],\mu)$
gives us back the usual product of functions. General associative
enhancements of general monoids $(\Cot P,\mu)$ corresponding to non-zero
Poisson structures should yield star-products on $C^{\infty}(P)[[\hbar]]$
upon asymptotic expansion. The first non trivial case is given by
the integral representation of the Moyal product on $\R^{2n}$ endowed
with its canonical symplectic form. 

This is to be related to a quantization program for Poisson manifolds
relying on asymptotic integrals of the Moyal type and on the symplectic
groupoid of the Poisson manifold (see \cite{Karasev1986,Weinstein1990,Zakrzewski1990}).
In this approach, a star-product should be obtained from the asymptotic
expansion of a semiclassical Fourier integral operator whose phase
is the generating function of the symplectic groupoid. In the special
case of symplectic symmetric spaces, similar integrals were worked
out in \cite{UniV,CR2005,RA2004,Rios2008}.
For the linear Poisson
structures, the paper \cite{AD2002} gives a integral version of Kontsevich's
star-product, where the generating function is the Baker-Campbell-Hausdorff
formula of the associated Lie algebra. In \cite{CDF2005}, it has
been shown that this latter generating function is exactly the generating
function of the corresponding symplectic groupoid and that, more generally,
generating functions of the local symplectic groupoid for Poisson
structures on open subsets of $\R^{n}$, can be extracted from the
tree-level part of Kontsevich star-product (\cite{Kontsevich2003}). 

In the same spirit,  Wagemann and one of the authors \cite{Leibniz} start by showing that the Gutt star-product 
can be understood as the quantization in the above sense of a  symplectic micromorphism, and then they extend this construction to quantize Leibniz algebras. 
Both in \cite{Leibniz}  and \cite{MomMap}, formal deformation quantizations are obtained from symplectic micromorphism quantization by taking Feynman expansions 
of the corresponding oscillatory integrals. 

\subsection{Quantization of symmetries}

Suppose now that we have a general symmetry\begin{eqnarray*}
\rho:\Cot P\otimes\Cot Q & \longrightarrow & \Cot Q,\end{eqnarray*}
where $(\Cot P,\mu)$ is a general monoid with induced Poisson structure
$\Pi$ on $P$, acting on $\Cot Q$. One can show that the core map
of $\rho$ is $J_{|Z_{Q}}\times\id_{Q}$, where $J:\Cot Q\rightarrow P$
is a Poisson map germ around the zero section of $\Cot Q$, which
can be considered as the momentum map of the the action. An enhancement
of $\rho$ is thus a half-density germ\[
[b]\in|\Omega_{\hbar}|^{\frac{1}{2}}\big(J_{|Q}^{*}(\Cot P)\oplus\Cot Q\big)\]
around the zero section. Quantizing this data, we obtain an operator\begin{eqnarray*}
Q_{\hbar}([b],\rho):\xH_{P}\otimes\xH_{Q} & \longrightarrow & \xH_{Q},\end{eqnarray*}
which is a representation of the quantum algebra $(\xH_{P},\star_{\hbar})$
on $\mathcal{H}_{Q}$ provided that \begin{eqnarray*}
[b]\circ([1]\otimes[b]) & = & [b]\circ([a]\otimes[1]).\end{eqnarray*}
This approach has been used in \cite{MomMap},  \cite{ActQuant}, and \cite{GSystems} to quantize 
group actions through their cotangent lifts as well as to quantize their momentum maps. 

\subsection{Quantization of the classical flow}

We want to quantize the symplectic micromorphism $\rho_{H}$ associated
to the classical flow $\Psi_{t}^{H}$ on $\Cot\R^{n}$ of a hamiltonian
$H:\Cot\R^{n}\rightarrow\R$ as described in paragraph \ref{sub:Classical-flow}.
Since, in general, $Q_{\hbar}(\Cot\R^{n})$ can be identified with
the space of $L^{2}$-functions on $\R^{n}$, we obtain, after quantization,
a linear operator\begin{eqnarray*}
Q_{\hbar}([a],\rho_{H}):L^{2}(\Energy)\otimes L^{2}(\R^{n}) & \longrightarrow & L^{2}(\R^{n})\end{eqnarray*}
for any enhancement $a=a(t,p_{1},x_{2})$ of the symplectic micromorphism.
Now, we may define the following the operator on $L^{2}(\R^{n})$
by\begin{eqnarray*}
U_{t}^{a} & := & Q_{\hbar}([a],\rho_{H})(u_{t}\otimes\,\cdot\,),\end{eqnarray*}
where $u_{t}(E)=e^{-\frac{i}{\hbar}t_{0}E}$ is the state with {}``definite
time'' $t_{0}$ on the space of energies. An explicit computation
yields\begin{eqnarray*}
(U_{t}^{a}\Phi)(x) & := & \frac{1}{(2\pi\hbar)^{n}}\SCint\Phi(x_{1})a(t_{0},p_{1},x)e^{\frac{i}{\hbar}p_{1}x_{1}-S(t_{0},p_{1},x)}dp_{1}dx_{1},\end{eqnarray*}
where $S$ is the generating function of phase flow $\Psi_{H}^{t}$
solution of the Hamilton-Jacobi equation as explained in paragraph
\ref{sub:Classical-flow}. Then, a classical result of semi-classical
analysis can now be reformulated in terms of symplectic micromorphisms
and their enhancements: For any Hamiltonian $H:\Cot\R^{n}\rightarrow\R$,
there exists an enhancement of $([a],\rho_{H})$ such that $U_{t}^{a}$
is the propagator, modulo $\hbar^{\infty}$, of the Schr\"odinger equation
with quantum hamiltonian given by the semi-classical pseudo-differential
operator\begin{eqnarray*}
\hat{H} & = & Q_{\hbar}([H],\id_{\Cot\R^{n}}),\end{eqnarray*}
where the germ $[H]$ is understood as an enhancement of the identity
map on $\Cot\R^{n}$ seen as symplectic micromorphism. Moreover, it
is easy to show that, $U_{0}^{a}=\id$ and $U_{t_{2}}^{a}\circ U_{t_{1}}^{a}=U_{t_{1}+t_{2}}^{a}$
 (when defined) is equivalent to $([a],\rho_{H})$ being an action
of energy monoid enhanced as in paragraph \ref{sub:Quant of the energy monoid}
on $\Cot Q$ thanks to the  fact that\begin{eqnarray*}
u_{t_{1}+t_{2}} & = & Q_{\hbar}([1],\mu_{\Energy})(u_{t_{1}}\otimes u_{t_{2}}),\end{eqnarray*}
and the module axioms:\begin{eqnarray*}
(U_{t_{2}}^{a}\circ U_{t_{1}}^{a})\Phi & = & \Big(Q_{\hbar}([a],\rho_{H})\circ\big(\id\otimes Q_{\hbar}([a],\rho_{H})\big)\Big)(u_{t_{2}}\otimes u_{t_{1}}\otimes\Phi),\\
 & = & Q_{\hbar}([a],\rho_{H})\circ\Big(\big(Q_{\hbar}([1],\mu_{\Energy})(u_{t_{2}}\otimes u_{t_{1}})\big)\otimes\Phi)\Big),\\
 & = & U_{t_{1}+t_{2}}^{a}\Phi.\end{eqnarray*}

\end{document}